\documentclass[12pt,oneside,english,reqno]{amsart}
\usepackage[T1]{fontenc}
\usepackage[latin9]{inputenc}
\usepackage{amsthm}

\makeatletter

\usepackage{mathtools}
\mathtoolsset{showonlyrefs}
\numberwithin{equation}{section}
\numberwithin{figure}{section}

\usepackage{geometry}
\usepackage[svgnames]{xcolor}
\usepackage[bookmarksnumbered=true]{hyperref} 
\hypersetup{
	colorlinks = true,
	linkcolor = Blue,
	anchorcolor = blue,
	citecolor = Green,
	filecolor = blue,
	urlcolor = FireBrick
}

\usepackage{aliascnt}
\usepackage[nameinlink]{cleveref}

\numberwithin{equation}{section}
\numberwithin{figure}{section}

\theoremstyle{plain}
\newtheorem{theorem}{Theorem}[section]

\theoremstyle{definition}
\newaliascnt{definition}{theorem}
\newtheorem{definition}[definition]{Definition}
\aliascntresetthe{definition}
\crefname{definition}{Definition}{Definitions}

\theoremstyle{definition}
\newaliascnt{question}{theorem}

\aliascntresetthe{question}
\crefname{question}{Question}{Questions}

\theoremstyle{definition}
\newaliascnt{remark}{theorem}
\newtheorem{remark}[remark]{Remark}
\aliascntresetthe{remark}
\crefname{remark}{Remark}{Remarks}

\theoremstyle{plain}
\newaliascnt{corollary}{theorem}

\aliascntresetthe{corollary}
\crefname{corollary}{Corollary}{Corollaries}

\theoremstyle{plain}
\newaliascnt{lemma}{theorem}
\newtheorem{lemma}[lemma]{Lemma}
\aliascntresetthe{lemma}
\crefname{lemma}{Lemma}{Lemmas}

\theoremstyle{plain}
\newaliascnt{proposition}{theorem}
\newtheorem{proposition}[proposition]{Proposition}
\aliascntresetthe{proposition}
\crefname{proposition}{Proposition}{Propositions}

\theoremstyle{definition}
\newaliascnt{example}{theorem}
\newtheorem{example}[example]{Example}
\aliascntresetthe{example}
\crefname{example}{Example}{Examples}

\theoremstyle{definition}
\newaliascnt{assumption}{theorem}

\aliascntresetthe{assumption}
\crefname{assumption}{Assumption}{Assumptions}

\theoremstyle{definition}
\newaliascnt{problem}{theorem}

\aliascntresetthe{problem}
\crefname{problem}{Problem}{Problems}

\newtheorem*{definition*}{Definition}
\newtheorem*{proposition*}{Proposition} 
\newtheorem*{remark*}{Remark}
\newtheorem*{example*}{Example}
\newtheorem*{problem*}{Problem}
\newtheorem*{question*}{Question}
\newtheorem*{theorem*}{Theorem}
\newtheorem*{lemma*}{Lemma}
\newtheorem*{corollary*}{Corollary}
\newtheorem*{Acknowledgments*}{Acknowledgments}

\renewenvironment{proof}[1][\proofname]{\medskip \noindent {\bfseries #1. }}{\hfill \qedsymbol\medskip}

\usepackage{bbm}
\usepackage{mathrsfs}
\usepackage{comment}
\usepackage{amsmath}
\usepackage{amssymb}

\DeclareRobustCommand{\SkipTocEntry}[5]{}

\newcommand{\mR}{\mathbb{R}}   
\newcommand{\mZ}{\mathbb{Z}}   
\newcommand{\mN}{\mathbb{N}}   
\newcommand{\mH}{\mathbb{H}}

\newcommand{\abs}[1]{\lvert #1 \rvert}  
\newcommand{\norm}[1]{\lVert #1 \rVert}  
\newcommand{\br}[1]{\langle #1 \rangle}  

\newcommand{\mD}{\mathbb{D}}

\newcommand{\mM}{\mathcal{M}}

\newcommand{\mP}{\mathcal{P}}
\newcommand{\mG}{\mathcal{G}}

\newcommand{\id}{\mathrm{Id}}

\newcommand{\bfi}{\mathrm{i}}

\newcommand{\ehat}{\,\widehat{\rule{0pt}{6pt}}\,}

\newcommand{\p}{\partial}

\DeclareMathOperator{\dist}{dist}
\DeclareMathOperator{\supp}{supp}

\newcommand{\rmd}{\mathrm{d}}

\usepackage{babel}

\newcommand{\Rn}{\mathbb{R}^n}

\renewcommand{\norm}[1]{\lVert #1 \rVert}
\renewcommand{\abs}[1]{\lvert #1 \rvert}

\makeatother

\usepackage{orcidlink}

\usepackage{babel}
\begin{document}
\begin{sloppypar}

\title[Inverse problems for the spectral fractional Laplacian]{Inverse problems for the spectral fractional Laplacian with inhomogeneous Dirichlet boundary data}

\author{Ravi Shankar Jaiswal\,\orcidlink{0009-0003-5845-0485}}
\address{Department of Mathematics, Southern University of Science and Technology, Xueyuan Avenue, Shenzhen, Guangdong, China 518055}
\email{ravi@sustech.edu.cn}
\author{Pu-Zhao Kow\,\orcidlink{0000-0002-2990-3591}}
\address{Department of Mathematical Sciences, National Chengchi University, Taipei 116, Taiwan}
\email{pzkow@g.nccu.edu.tw}

\author{Suman Kumar Sahoo\,\orcidlink{0000-0002-6459-1597}}
\address{Department of Mathematics
IIT Bombay, Powai,
Mumbai, India 400076}
\email{suman@math.iitb.ac.in}

\begin{abstract}
In this paper, we study the spectral fractional Laplacian with inhomogeneous Dirichlet boundary data, following the framework of \cite{APR18SpectralFractionalLaplacian}. Our contributions are twofold: first we introduce a Dirichlet-to-Neumann map for this operator and analyze an associated inverse problem; and second we establish an additional density result for the spectral fractional Laplacian.
\end{abstract}


\maketitle

\tableofcontents

\section{Introduction}

The study of the classical inverse problem dates back to the seminal work of Calder{\'o}n \cite{Calderon1980} in $1980$, where he posed the question of whether one can determine the electrical conductivity of a medium from boundary measurements of current and voltage. In the same work, Calder{\'o}n also proved the linearized version of the problem using complex geometrical optics (CGO) solutions.
Since then, substantial developments have been made in the mathematical theory of inverse problems, driven by a wide range of applications, including medical imaging and seismic imaging. We refer the reader to the surveys \cite{Uhl_eip_survey,Uhlmann_survey} and the monograph \cite{FSU25CalderonIntro} for further results in this direction. 

The inverse problem for the fractional Laplace operator is a relatively recent research topic compared to the corresponding problem for the classical (non-fractional) Schr{\"o}dinger operator. Nevertheless, it has attracted significant attention in a short period of time and a substantial body of literature has emerged. We refer to \cite{GSU20Calderon,GRSU20Reconstruction,RS18Instability,RS20Calderon} and the references therein.
A key feature exploited in these works is the nonlocal nature of the fractional Laplace operator, which enjoys a strong unique continuation property (UCP), see, for example, \cite{GSU20Calderon,Rul15unique,Ros16FractionalLaplacianSurvey}. In particular, the work \cite{GSU20Calderon} employed the Runge approximation property (a quantitative form of UCP) to solve the fractional inverse problem. This approach was further extended and developed in the subsequent works, see \cite{BCR25_siam,Cov20FractionalConductivity,KLW21CalderonFractionalWave,CMR20unique,CJTZ_Siam,BTG_transectionsams,FTKG_jdg,TG_revista} and the references therein.

In this paper, we study an inverse problem for the spectral fractional Laplace operator (see \Cref{def:spectral-fractional-laplacian} below), and we provide a solution to its linearized version near zero potential. To this end, let $\Omega\subset \Rn$ ($n\ge 2$) be a bounded domain with smooth boundary.
We consider the following Dirichlet problem for the Schr{\"o}dinger equation involving the spectral fractional Laplacian: 
\begin{equation}
\left( (-\Delta_{D})^{s}-q \right)u=0 \text{ in $\Omega$} ,\quad u|_{\partial\Omega} = g, \label{eq:sch-main}
\end{equation}
assuming that 0 is not a Dirichlet eigenvalue of $\left( (-\Delta_{D,0})^{s} - q \right)$. 

We will later give a precise definition of the spectral fractional Laplacians $(-\Delta_{D})^{s}$ and $(-\Delta_{D,0})^{s}$ in suitable Hilbert spaces. These operators coincide on $C_{c}^{\infty}(\Omega)$, see \cite[Proposition~2.4]{APR18SpectralFractionalLaplacian}. 
The Dirichlet-to-Neumann (DN) map associated with \eqref{eq:sch-main} is defined as 
\begin{equation*}
\Lambda_{q} : H^{s-\frac{1}{2}}(\partial\Omega) \rightarrow H^{\frac{1}{2}}(\partial\Omega), 
\end{equation*}
with the precise definition given in \eqref{eq:s-DN-map-definition}. If $\norm{q}_{L^{\infty}(\Omega)}$ is sufficiently small, the solution $u$ of \eqref{eq:sch-main} can be approximated by the unique solution $\tilde{u}$ of 
\begin{subequations} \label{eq:Born-approximation}
\begin{equation}
(-\Delta_{D,0})^{s}\tilde{u} = qu_{0} \text{ in $\Omega$} ,\quad \tilde{u}|_{\partial\Omega}=0, 
\end{equation}
where $u_{0}$ denotes the unique solution of  
\begin{equation}
(-\Delta)_{D}^{s}u_{0}=0 \text{ in $\Omega$} ,\quad u_{0}|_{\partial\Omega}=g. 
\end{equation}
\end{subequations} 
This is known as the Born approximation, see \Cref{sec:Born-approximation}. 
The DN map associated with \eqref{eq:Born-approximation} defines a bounded linear operator 
\begin{equation}
\rmd \Lambda : L^{\infty}(\Omega) \rightarrow \mathcal{L}(H^{s-\frac{1}{2}}(\partial\Omega),H^{\frac{1}{2}}(\partial\Omega)), \label{eq:linearized-s-DN-map-introduction}
\end{equation}
where $\mathcal{L}(X,Y)$ denotes the space of bounded linear operators from $X$ to $Y$. 
Note that $\rmd \Lambda$ is precisely the Fr{\'e}chet derivative of the non-linear map $q\mapsto \Lambda_{q}^{s}$ at $q=0$, see \Cref{prop:Frechet-derivative-s-DN}. 
The main focus of this paper is the injectivity of the operator \eqref{eq:linearized-s-DN-map-introduction}.

\begin{theorem}\label{thm:main}
Let $n\ge 2$ be an integer, and let $\Omega\subset\mR^{n}$ be a bounded $C^{1,\alpha}$ domain for some $\alpha>\frac{1}{2}$. Fix $\frac{1}{2}<s<1$, and let $q^{(1)},q^{(2)}\in L^{\infty}(\Omega)$ satisfy 
\begin{equation}
\norm{q^{(j)}}_{L^{2}(\Omega)}\le M\quad\text{for all $j=1,2$.} \label{eq:apriori}
\end{equation}
If 
\begin{equation}
(\rmd\Lambda^{s}[q^{(1)}])(g)=(\rmd\Lambda^{s}[q^{(1)}])(g)\quad\text{for all $g\in C^{\infty}(\partial\Omega)$,} \label{eq:DN-map-data}
\end{equation}
then $q^{(1)}\equiv q^{(2)}$. Moreover, there exists a constant $C=C(n,\Omega,M)$ such that 
\[
\norm{\chi_{\Omega}(q^{(1)}-q^{(2)})}_{H^{-1}(\mR^{n})}\le C w(\norm{\rmd\Lambda^{s}[q^{(1)}] - \rmd\Lambda^{s}[q^{(2)}]}_{*}),
\]
where the operator norm $\norm{\cdot}_{*}$ is given by $\norm{\cdot}_{*}=\norm{\cdot}_{H^{\frac{1}{2}}(\partial\Omega)\rightarrow H^{-\frac{1}{2}}(\partial\Omega)}$, and $w$ is a modulus of continuity given by 
\[
w(t) \lesssim \abs{\operatorname{log}t}^{-1}, \quad 0 < t < 1/e.\]
\end{theorem}


\begin{remark} 
Another variant of the Dirichlet problem of the Schr\"{o}dinger equation involves the Fourier fractional Laplacian $(-\Delta_{F})^{s}$: 
\begin{equation*}
\left( (-\Delta_{F})^{s} - q \right) u = 0 \text{ in $\Omega$} ,\quad u|_{\mR^{n}\setminus\overline{\Omega}} = g, 
\end{equation*}
which is clearly different from \eqref{eq:sch-main}. 
Given any open sets $W_{1},W_{2} \subset \mR^{n}\setminus\overline{\Omega}$, the reconstruction of $q$ from the exterior data $(\tilde{u}|_{W_{1}},(-\Delta_{F})^{s}\tilde{u}|_{W_{2}})$ has been extensively studied, see, for example, \cite{GSU20Calderon,GRSU20Reconstruction,RS20Calderon}. 
In a recent paper \cite{TG_revista}, Ghosh further showed that $q$ can also be uniquely determined from the data $\left( \tilde{u}|_{W_{1}},\frac{\tilde{u}}{\dist\,(\cdot,\partial\Omega)^{s}}|_{\Sigma} \right)$, where $\Sigma$ is a non-empty open subset of $\partial\Omega$. 
\end{remark} 

\begin{remark} 
Another variant of the Dirichlet problem was considered in \cite{AD17SpectralFractionalLaplacian}. There, the authors study 
\begin{equation}
(-\Delta_{D,0})^{s} = \mu \text{ in $\Omega$} ,\quad \left. \frac{u}{h_{1}} \right|_{\partial\Omega} = g, \label{eq:Hadamard}
\end{equation}
where $h_{1}$ is a reference function that is bounded above and below by constant multiples of $\dist \, (\cdot,\partial\Omega)$. This formulation is also different from \eqref{eq:sch-main}. Under some assumptions, the problem \eqref{eq:Hadamard} is well-posed in the sense of Hadamard. 
\end{remark}

We now state our second main result, up to the natural gauge invariance. To this end, we first discuss the corresponding gauge class. Let $w\in C^{\infty}(\overline{\Omega})$ satisfy $w=\partial_{\nu}w=0$ on $\partial\Omega$. Define 
\begin{equation}
\theta^{2}=w\id_{n} ,\quad \theta^{1}=2\nabla w \quad \text{and} \quad \theta^{0}=\Delta w. \label{eq:natural-gauge} 
\end{equation}
Then, for any sufficiently smooth function $u$, one can write 
\begin{equation}
\Delta(wu) = \theta^{2}:\nabla^{\otimes 2}u + \theta^{1}\cdot\nabla u + \theta^{0}u, \label{eq:product-rule}
\end{equation}
where $\id_{n}$ is the $n\times n$ identity matrix, $\nabla^{\otimes 2}u$ is the Hessian matrix of $u$ and we use the convention 
\begin{equation*}
A:B = \sum_{i,j=1}^{n}A_{ij}B_{ij} \quad \text{for matrices $A=(A_{ij})$ and $B=(B_{ij})$.}
\end{equation*}
We also adopt the notations $(u\otimes v)_{ij}=u_{i}v_{j}$ and $\delta_{ij}=(I_{n})_{ij}$. For any symmetric matrix $A$, note that 
\begin{equation*}
A:(u\otimes u) = u^{\intercal}Au=Au\cdot u 
\end{equation*}
for any vector $u$. 
Using \eqref{eq:product-rule}, one can easily check that 
\begin{equation}
\int_{\Omega} (\theta^{2}:\nabla^{\otimes 2}u + \theta^{1}\cdot\nabla u + \theta^{0}u)v\,\rmd x = 0 \label{eq:integral-identity}
\end{equation}
for all harmonic function $v\in C^{2}(\overline{\Omega})$. Our next result concerns the recovery of the coefficients $\theta^{2},\theta^{1}$ and $\theta^{0}$ up to the natural gauge \eqref{eq:natural-gauge}:


\begin{theorem}\label{th_desnity}
Let $n\ge 3$, $0<s<1$, and $\Omega$ be a smooth bounded domain. If $\theta^{2}\in(C^{\infty}(\overline{\Omega}))^{n\times n}$, $\theta^{1}\in(C^{\infty}(\overline{\Omega}))^{n}$ and $\theta^{0}\in C^{\infty}(\overline{\Omega})$ satisfy
\begin{equation*}
\theta^{j}=\partial_{\nu}^{k}\theta^{j}=0 \text{ on $\partial\Omega$} \quad \text{for all $j=0,1,2$ and for all $k\le 11$.}
\end{equation*}
If \eqref{eq:integral-identity} holds for all $u\in C^{\infty}(\overline{\Omega})$ and $v\in C^{\infty}(\overline{\Omega})$ such that 
\begin{equation*}
\Delta^{2}u = 0 \quad \text{and} \quad (-\Delta_{D})^s v=0 \quad \text{in $\Omega$,} 
\end{equation*}
then we obtain the exact gauge \eqref{eq:natural-gauge} for some $w\in C^{\infty}(\overline{\Omega})$ such that $w=\partial_{\nu}w=0$ on $\partial\Omega$.  
In the case where $n=2$, if we additionally assume that $\theta^0=0$ and ${\rm tr}\,(\theta^2)=0$, then we conclude $\theta^{1}=0$ and $\theta^{2}=0$. 
\end{theorem}

\begin{remark}
The above discussion shows that our result is sharp for $n\ge 3$, however, we are unable to identify the exact gauge \eqref{eq:natural-gauge} in two dimensions.
It should therefore be noted that if one replaces the condition $
\Delta v=0$ by $\Delta^{2}v=0$, then it was shown in \cite[Theorem~2.1]{SS_linearized} that $\theta^{2}=0$, $\theta^{1}=0$ and $\theta^{0}=0$. In other words, \Cref{th_desnity} can be regarded as an extension of \cite[Theorem~2.1]{SS_linearized}. 
\end{remark}

\begin{remark}
We emphasize that the condition $(-\Delta_{D})^{s}v=0$ for any $0<s<1$ is essentially equivalent to $\Delta v=0$, see \eqref{eq:Poisson-equivalence} below. One might also ask whether condition $(-\Delta_{D})^{s}u=0$ is replaced by a condition involving $(-\Delta_{D})^{\gamma}$ with exponent $1<\gamma<2$. However, there appears to be no natural definition of $(-\Delta_{D})^{\gamma}$ in this range with an inhomogeneous boundary condition, see \Cref{def:extension-high-order} below. 
\end{remark}



\section{\label{sec:Preliminaries}Preliminaries}

\addtocontents{toc}{\SkipTocEntry}
\subsection{Fractional Sobolev space}

We first introduce some well-known fractional Sobolev space following \cite{APR18SpectralFractionalLaplacian,KK22LewyStampacchia}, see also \cite{CS16FractionalLaplacian,LM72FractionalSobolev1,McL00EllipticSystems,Mik11EllipticSystems,NOS15FractionalDiffusion}. 

Let $\Omega$ be a bounded Lipschitz domain in $\mR^{n}$. For $0<s<1$, let $H^{s}(\Omega)$ be the fractional Sobolev space equipped with the norm 
\[
\norm{\cdot}_{H^{s}(\Omega)}^{2}:=\norm{\cdot}_{L^{2}(\Omega)}^{2}+[\cdot]_{H^{s}(\Omega)}^{2},
\]
where the Gagliardo seminorm $[\cdot]_{H^{s}(\Omega)}$ is defined by 
\[
[v]_{H^{s}(\Omega)}^{2} := \int_{\Omega} \int_{\Omega} \frac{\abs{v(x)-v(z)}^{2}}{\abs{x-z}^{n+2s}}\,\rmd x \,\rmd z.
\]
For $0<s<1$ with $s\neq\frac{1}{2}$, we define $H_{0}^{s}(\Omega)$ be the completion of $C_{c}^{\infty}(\Omega)$ with respect to $\|\cdot\|_{H^{s}(\Omega)}$. When $s=\frac{1}{2}$, the Lions-Magenes space\footnote{Some authors use the notation $H_{00}^{\frac{1}{2}}(\Omega)$ to represent the Lions-Magenes spacce.} $H_{0}^{\frac{1}{2}}(\Omega)$ is defined by 
\[
H_{0}^{\frac{1}{2}}(\Omega):= \left\{ \begin{array}{l|l} v\in H^{\frac{1}{2}}(\Omega) & {\displaystyle \int_{\Omega}\frac{\abs{v(x)}^{2}}{\dist\,(x,\partial\Omega)}\,\rmd x<\infty}\end{array} \right\},
\]
equipped with the norm 
\[
\norm{v}_{H_{0}^{\frac{1}{2}}(\Omega)}^{2}:=\norm{v}_{H^{\frac{1}{2}}(\Omega)}^{2}+\int_{\Omega}\frac{\abs{v(x)}^{2}}{\dist\,(x,\partial\Omega)}\,\rmd x.
\]
In particular, we have 
\begin{align*}
H_{0}^{s}(\Omega) & =H^{s}(\Omega)\quad\text{for all $0<s<\frac{1}{2}$,}\\
H_{0}^{s}(\Omega) & \subsetneq H^{s}(\Omega)\quad\text{for all $\frac{1}{2}\le s<1$. }
\end{align*}

Let $H^{-s}(\Omega)$ be the dual space of $H_{0}^{s}(\Omega)$. It is well-known that there exists a sequence of $H_{0}^{1}(\Omega)$-eigenvalues $\{\lambda_{k}\}_{k=0}^{\infty}$ of the Dirichlet Laplacian $-\Delta$, with corresponding eigenfunctions $\{\phi_{k}\}_{k=0}^{\infty}\subset H_{0}^{1}(\Omega)\cap C^{\infty}(\Omega)$. Moreover, the eigenfunctions $\{\phi_{k}\}_{k=0}^{\infty}$ form an orthonormal basis of $L^{2}(\Omega)$. Accordingly, for each $\gamma\in\mR$, we can define the following fractional-order Sobolev space 
\[
\mH^{\gamma}(\Omega):= \left\{ \begin{array}{l|l} v={\displaystyle \sum_{k=0}^{\infty}}v_{k}\varphi_{k}\in L^{2}(\Omega) & \norm{v}_{\mH^{\gamma}(\Omega)}^{2}={\displaystyle \sum_{k=0}^{\infty}\lambda_{k}^{\gamma}}\abs{\br{v,\varphi_{k}}_{\Omega}}^{2}<\infty\end{array} \right\}.
\]
It is well-known that 
\begin{equation}
\mH^{s}(\Omega)\equiv H_{0}^{s}(\Omega)\quad\text{for all $0<s<1$}\quad\text{(with equivalent norms)},\label{eq:norm-equivalence1}
\end{equation}
see, e.g., \cite{APR18SpectralFractionalLaplacian,MN14FractionalLaplacian1,MN14FractionalLaplacian2}.

\addtocontents{toc}{\SkipTocEntry}
\subsection{Spectral fractional Laplacian}\label{sec:spectral laplacian}

We now define the (homogeneous) spectral fractional Laplacian $(-\Delta_{D,0})^{s}:\mH^{2s}(\Omega)\rightarrow L^{2}(\Omega)$ by 
\[
(-\Delta_{D,0})^{s}v:=\sum_{k=1}^{\infty}\lambda_{k}^{s} \br{v,\phi_{k}}_{\Omega}\phi_{k}\quad\text{for all $v\in\mathbb{H}^{2s}(\Omega)$,}
\]
where $\br{\cdot,\cdot}_{\Omega}$ is the $L^2$ inner product in $\Omega$, and it is easy to see that 
\begin{equation}
\norm{(-\Delta_{D,0})^{s}\cdot}_{L^{2}(\Omega)} \equiv \norm{\cdot}_{\mH^{2s}(\Omega)}.\label{eq:norm-equivalence2}
\end{equation}
In particular, $(-\Delta_{D,0})^{s}:\mH^{s}(\Omega)\rightarrow\mH^{-s}(\Omega)$ is also a bounded linear operator. We now define 
\[
-\Delta_{D}v := \sum_{k=1}^{\infty} \lambda_{k} \left(\br{v,\phi_{k}}_{\Omega}+\lambda_{k}^{-1}\br{v,\partial_{\nu}\phi_{k}}_{\partial\Omega}\right)\phi_{k}\quad\text{for all $v\in H^{1}(\Omega)$,}
\]
where $\br{\cdot,\cdot}_{\partial\Omega}$ is the distribution pairing in $\partial\Omega$. Using \cite[Proposition~2.3]{APR18SpectralFractionalLaplacian}, for each $v\in C^{\infty}(\overline{\Omega})$ we know that $-\Delta_{D}v=-\Delta v$ a.e. in $\Omega$. We now introduce the spectral fractional Laplacian with inhomogeneous Dirichlet boundary data as in \cite[Definition~2.3]{APR18SpectralFractionalLaplacian}. 

\begin{definition}\label{def:spectral-fractional-laplacian} 
We define the (inhomogeneous Dirichlet) \emph{spectral fractional Laplacian} by 
\[
(-\Delta_{D})^{s}v := \sum_{k=1}^{\infty}\lambda_{k}^{s} \left(\br{v,\phi_{k}}_{\Omega}+\lambda_{k}^{-1}\br{v,\partial_{\nu}\phi_{k}}_{\partial\Omega}\right)\phi_{k}\quad\text{for all $v\in\mathbb{D}^{2s}(\Omega)$,}
\]
so that $(-\Delta_{D})^{s}:\mD^{2s}(\Omega)\rightarrow L^{2}(\Omega)$ is a linear bounded operator, where $\mD^{\gamma}(\Omega)$ is given by 
\[
\mD^{\gamma}(\Omega):= \left\{ \begin{array}{l|l} v\in L^{2}(\Omega) & {\displaystyle \sum_{k=1}^{\infty}\lambda_{k}^{\gamma}\left(\br{v,\phi_{k}}_{\Omega}+\lambda_{k}^{-1}\br{v,\partial_{\nu}\phi_{k}}_{\partial\Omega}\right)^{2}<\infty}\end{array} \right\}
\]
\end{definition}

\begin{remark}
In the paragraph following \cite[Definition~2.3]{APR18SpectralFractionalLaplacian}, it is further shown that the operator $(-\Delta_{D})^{s}$ can be extended to an operator mapping from $\mD^{s}(\Omega)$ to $\mH^{-s}(\Omega)$. 
\end{remark}

\begin{remark}\label{def:extension-high-order}
If $u\in C^{\infty}(\overline{\Omega})$ satisfies $\Delta u\in \mD^{2s}(\Omega)$, then $(-\Delta_{D})^{s}(-\Delta u)\in L^{2}(\Omega)$ and 
\begin{equation*}
\begin{aligned}
& (-\Delta_{D})^{s}(-\Delta u) = \sum_{k=1}^{\infty} \left(-\lambda_{k}^{s}\int_{\Omega}\Delta u\phi_{k}\,\rmd x - \lambda_{k}^{s-1}\int_{\partial\Omega}\Delta u \partial_{\nu}\phi_{k}\,\rmd S \right) \phi_{k} \\ 
&\quad = \sum_{k=1}^{\infty} \left( \lambda_{k}^{s} \left( \lambda_{k}\int_{\Omega}\phi_{k}u\,\rmd x + \int_{\partial\Omega} u\partial_{\nu}\phi_{k}\,\rmd S\right) - \lambda_{k}^{s-1}\int_{\partial\Omega}\Delta u\partial_{\nu}\phi_{k}\,\rmd S\right) \phi_{k} \\ 
& \quad = \sum_{k=1}^{\infty} \lambda_{k}^{s+1} \left( \br{u,\phi_{k}}_{\Omega} + \lambda_{k}^{-1}\br{u,\partial_{\nu}\phi_{k}}_{\partial\Omega} - \lambda_{k}^{-2}\br{\Delta u,\partial_{\nu}\phi_{k}}_{\partial\Omega} \right) \phi_{k}. 
\end{aligned}
\end{equation*}
On the other hand, if $u\in \mD^{2s+2}(\Omega)$, then 
\begin{equation*}
(-\Delta)(-\Delta_{D})^{s}u = \sum_{k=1}^{\infty}\lambda_{k}^{s+1}\left(\br{u,\phi_{k}}_{\Omega} + \lambda_{k}^{-1}\br{u,\partial_{\nu}\phi_{k}}_{\partial\Omega} +\lambda_k^{-(s+1)}\br{ (-\Delta_D)^s u,\partial_{\nu}\phi_{k}}_{\partial\Omega}  \right)\phi_{k}. 
\end{equation*}
Consequently, it does not seem natural to extend the operator $(-\Delta_{D})^{\gamma}$ to exponents $\gamma>1$ under inhomogeneous boundary conditions. 
\end{remark}

\addtocontents{toc}{\SkipTocEntry}
\subsection{Traces and integration by parts}

We now assume that $\partial\Omega\in C^{1,\alpha}$ for some $\alpha>\frac{1}{2}$. Using \cite[Lemma~3.1]{APR18SpectralFractionalLaplacian} or \cite[Lemma~6.3]{GM11SelfAdjointLaplacian}, the Neumann trace operator 
\begin{equation}
\partial_{\nu}:H_{0}^{1}(\Omega)\cap H^{2}(\Omega)\rightarrow H^{\frac{1}{2}}(\partial\Omega) \label{eq:Neumann-trace-operator}
\end{equation}
is a well-defined linear bounded surjective operator with a linear bounded right inverse. In addition, we have 
\begin{equation}
\ker(\partial_{\nu})=H_{0}^{2}(\Omega).\label{eq:ker-normal-der}
\end{equation}
From \cite[(8.10) and Definition~8.9]{GM11SelfAdjointLaplacian}, we know that the following integration by parts formula is a special case of \cite[Theorem~3.1]{APR18SpectralFractionalLaplacian}. 

\begin{lemma} \label{lem:integration-by-parts}
Let $\Omega$ be a bounded $C^{1,\alpha}$ domain for some $\alpha>\frac{1}{2}$, and let $0<s<1$. Given any $u\in\mD^{2s}(\Omega)$ with $u|_{\partial\Omega}\in H^{-\frac{1}{2}}(\partial\Omega)$ and $v\in\mH^{2s}(\Omega)$, we have the following integration by parts formula: 
\[
\br{u,\partial_{\nu}w_{v}}_{\partial\Omega} = \left((-\Delta_{D})^{s}u,v\right)_{L^{2}(\Omega)}-\left(u,(-\Delta_{D,0})^{s}v\right)_{L^{2}(\Omega)},
\]
where $w_{v}\in\mH^{2}(\Omega)$ is the unique solution of $(-\Delta_{D,0})^{1-s}w_{v}=v$ in $\Omega$. In this case, $\br{\cdot,\cdot}_{\partial\Omega}$ is simply the $H^{-\frac{1}{2}}(\partial\Omega)\times H^{\frac{1}{2}}(\partial\Omega)$ duality pair. 
\end{lemma}

\begin{remark}\label{rem:bounded-Neumann-data}
In particular $w_{v}\in H^{2}(\Omega)\cap H_{0}^{1}(\Omega)$ and we have 
\[
\norm{w_{v}}_{H^{2}(\Omega)\cap H_{0}^{1}(\Omega)} \le C \norm{\Delta w_{v}}_{L^{2}(\Omega)} = C \norm{(-\Delta_{D,0})^{s}v}_{L^{2}(\Omega)} = C \norm{v}_{\mH^{2s}(\Omega)},
\]
see \cite[(3.7)]{APR18SpectralFractionalLaplacian}. 
\end{remark}

\addtocontents{toc}{\SkipTocEntry}
\subsection{Definition of DN map and its linearization\label{sec:Born-approximation}}

We first recall some facts related to classical Schr\"{o}dinger equation (i.e. \eqref{eq:sch-main} corresponds to $s=1$), which reads 
\begin{equation}
\left( -\Delta -q \right)u=0 \text{ in $\Omega$} ,\quad u|_{\partial\Omega} = g.  \label{eq:sch-main-classical}
\end{equation}
Suppose that 0 is not an eigenvalue of \eqref{eq:sch-main-classical}. In this case, for each $g\in H^{\frac{1}{2}}(\partial\Omega)$, it is well-known that the normal derivative $\partial_{\nu}$ is well-defined in $H^{-\frac{1}{2}}(\partial\Omega)$ by the formula 
\begin{equation}
\left( \partial_{\nu}u,\phi \right)_{\partial\Omega} := \int_{\Omega} \left( \nabla u \cdot \nabla \tilde{\phi} + qu\tilde{\phi} \right) \, \rmd x \quad \text{for all $g,\phi\in H^{\frac{1}{2}}(\partial\Omega)$,} \label{eq:Definition-DN1}
\end{equation}
where $\tilde{\phi} \in H^{1}(\Omega)$ is any function with trace $\phi$ on $\partial\Omega$. The definition \eqref{eq:Definition-DN1} is independent of choices of $\tilde{\phi}$. It is well-known that the potential $q\in L^{\infty}(\Omega)$ is uniquely determined by the mapping $g \mapsto \partial_{\nu}u$, see \cite{SU87Calderon}. Let $\mP : H^{\frac{1}{2}}(\partial\Omega)\rightarrow H^{1}(\Omega)$ be the classical Poisson operator such that $\mP g$ is the unique solution of 
\begin{equation}
-\Delta \mP g=0  \text{ in $\Omega$} ,\quad \mP g|_{\partial\Omega} = g. \label{eq:linearized-linear1}
\end{equation}
In fact, the well-posedness of \eqref{eq:linearized-linear1} still holds for low regularity boundary data $g$. If $g\in H^{s-\frac{1}{2}}(\partial\Omega)$ with $s\in[0,1]$, using \cite[Lemma~4.1]{APR18SpectralFractionalLaplacian}, there exists a unique very-weak solution $\mP g\in H^{s}(\Omega)$ of \eqref{eq:linearized-linear1} in the sense of 
\[
\int_{\Omega} (\mP g)\left((-\Delta)\varphi\right) \,\rmd x = -\int_{\partial\Omega} g\partial_{\nu}\varphi \, \rmd S \quad\text{for all $\varphi\in H_{0}^{1}(\Omega)\cap H^{2}(\Omega)$.}
\]
Using \cite[Remark~4.1]{APR18SpectralFractionalLaplacian}, the well-posedness result can be extended to general Lipschitz domains when $s\in[\frac{1}{2},1]$. 

We now see that $\mP g$ does not carry any information about the potential $q \in L^{\infty}(\Omega)$ and 
\begin{equation*}
( -\Delta - q ) (u - \mP g) = q \mP g \text{ in $\Omega$} ,\quad (u - \mP g)|_{\partial\Omega}=0.
\end{equation*}
In fact the potential $q\in L^{\infty}(\Omega)$ is uniquely determined by the mapping $g \mapsto \partial_{\nu}(u - \mP g)$. This suggests us to consider an alternative definition of DN map as follows: 
\begin{equation*}
\Lambda_{q} : H^{\frac{1}{2}}(\partial\Omega) \mapsto H^{-\frac{1}{2}}(\partial\Omega) ,\quad \Lambda_{q}g := \partial_{\nu}(u - \mP g). 
\end{equation*}

We now define the corresponding DN map for the fractional Schr\"{o}dinger equation \eqref{eq:sch-main}. The following is a special case of \cite[Theorem~4.2]{APR18SpectralFractionalLaplacian}. 

\begin{lemma}[Existence and uniqueness] \label{lem:well-posedness}
Let $\Omega$ be a bounded $C^{1,\alpha}$ domain for some $\alpha>\frac{1}{2}$, and let $\frac{1}{2}<s<1$. Given any $f\in H^{-s}(\Omega)$ and $g\in H^{s-\frac{1}{2}}(\partial\Omega)$, there exists a unique solution $v\in H^{s}(\Omega)$ of 
\begin{equation}
(-\Delta_{D})^{s}v=f \text{ in $\Omega$} ,\quad v|_{\partial\Omega} = g. 
\end{equation}
In addition, there exists a positive constant $C=C(\Omega,s)$, which is independent of $v,f$ and $g$, such that 
\[
\norm{v}_{H^{s}(\Omega)}\le C \left( \|f\|_{H^{-s}(\Omega)} + \norm{g}_{H^{s-\frac{1}{2}}(\partial\Omega)} \right). 
\]
\end{lemma}

Accordingly, we may define the Poisson operator $\mP^{s}:H^{s-\frac{1}{2}}(\partial\Omega)\rightarrow H^{s}(\Omega)$ by $\mP^{s}g := u_{g}$, where $u_{g}=v$ is the unique function given in \Cref{lem:well-posedness} with $f\equiv 0$. It is worth-noting to mention \cite[Theorem~4.1]{APR18SpectralFractionalLaplacian} that 
\begin{equation}
\mP g = \mP^{s} g \quad\text{for all $g \in H^{\frac{1}{2}}(\partial\Omega)$.} \label{eq:Poisson-equivalence}
\end{equation}

From \eqref{eq:sch-main}, we now see that 
\begin{equation*}
\left( (-\Delta_{D,0})^{s} - q \right) (\tilde{u} - \mP^{s} g) = q \mP^{s} g \text{ in $\Omega$} ,\quad (\tilde{u} - \mP^{s} g)|_{\partial\Omega}=0.
\end{equation*}
Since 0 is not a Dirichlet eigenvalue of $\left( (-\Delta_{D,0})^{s} - q \right)$, for each $g \in H^{s-\frac{1}{2}}(\partial\Omega)$, there exists a unique solution $\tilde{v}[g] \in H_{0}^{s}(\Omega)$ of 
\begin{equation}
\left( (-\Delta_{D,0})^{s} - q \right) \tilde{v}[g] = q \mP^{s} g \text{ in $\Omega$} ,\quad \tilde{v}[g]|_{\partial\Omega}=0. \label{eq:v-tilde-g}
\end{equation}
By using \eqref{eq:norm-equivalence2}, one sees that $\tilde{v}[g] \in\mH^{2s}(\Omega)$. Accordingly, the DN map of \eqref{eq:linearized-linear1} can be defined by 
\begin{equation}
\Lambda_{q}^{s}g := \partial_{\nu} w_{\tilde{v}[g]} \quad \text{on $\partial\Omega$}, \label{eq:s-DN-map-definition}
\end{equation}
where $w_{v}$ is the function given in \Cref{lem:integration-by-parts}. 

If $\norm{q}_{L^{\infty}(\Omega)}$ is smaller than the first Dirichlet eigenvalue of $(-\Delta_{D,0})^{s}$, then 0 is not a Dirichlet eigenvalue of $\left( (-\Delta_{D,0})^{s} - q \right)$. In this case, the formula above suggests us to approximate $\tilde{v}[g]$ by the unique solution $v[g] \in H_{0}^{s}(\Omega) \cap \mH^{2s}(\Omega)$ of 
\begin{equation}
(-\Delta_{D,0})^{s} v[g] = q \mP^{s} g \text{ in $\Omega$} ,\quad v[g]|_{\partial\Omega}=0. \label{eq:Born-approx}
\end{equation}
By using \Cref{lem:well-posedness}, one sees that \eqref{eq:Born-approx} is well-posed \emph{for all} $q\in L^{\infty}(\Omega)$. Accordingly, we can define the mapping 
\begin{equation}
(\rmd \Lambda^{s}[q])g := \partial_{\nu} w_{v[g]} \quad \text{on $\partial\Omega$}, \label{eq:linearized-s-DN-map-definition}
\end{equation}
where $w_{v}$ is the function given in \Cref{lem:integration-by-parts}. 

By using the observation $\Lambda_{0}^{s}g = 0$ for all $g\in H^{s-\frac{1}{2}}(\partial\Omega)$, one can show the $\rmd \Lambda^{s}$ is the Fr\'{e}chet derivative of the non-linear functor $q \mapsto \Lambda_{q}^{s}$ at $q=0$. 

\begin{proposition} \label{prop:Frechet-derivative-s-DN}
Let $\Omega$ be a bounded $C^{1,\alpha}$ domain for some $\alpha>\frac{1}{2}$, let $\frac{1}{2}<s<1$, and let $\lambda(\Omega,s) > 0$ be the first Dirichlet eigenvalue of $(-\Delta_{D,0})^{s}$. Then there exists a constant $C = C(\Omega,s)$ such that 
\begin{equation}
\norm{\Lambda_{q}^{s} - \Lambda_{0}^{s} - \rmd \Lambda^{s}[q]}_{H^{s-\frac{1}{2}}(\partial\Omega) \rightarrow H^{\frac{1}{2}}(\partial\Omega)} \le C(\Omega,s) \norm{q}_{L^{\infty}(\Omega)}^{2}, \label{eq:Frechet}
\end{equation}
for all $q \in L^{\infty}(\Omega)$ with $\norm{q}_{L^{\infty}(\Omega)} \le \frac{1}{2} \lambda(\Omega,s)$. 
\end{proposition}

\begin{proof}
From \eqref{eq:v-tilde-g} and \eqref{eq:Born-approx}, we see that 
\begin{equation*}
(-\Delta_{D,0})^{s} \left( \tilde{v}[g] - v[g] \right) = q \tilde{v}[g], 
\end{equation*}
therefore from \eqref{eq:Neumann-trace-operator}, \Cref{rem:bounded-Neumann-data} and \eqref{eq:norm-equivalence2} we have 
\begin{equation}
\begin{aligned}
& \norm{\Lambda_{q}^{s}g - \Lambda_{0}^{s}g - (\rmd \Lambda^{s}[q])g}_{H^{\frac{1}{2}}(\partial\Omega)} \\
& \quad \le C(\Omega) \norm{\tilde{v}[g] - v[g]}_{\mH^{2s}(\Omega)} \le C(\Omega) \norm{q}_{L^{\infty}(\Omega)} \norm{\tilde{v}[g]}_{L^{2}(\Omega)}. 
\end{aligned} \label{eq:Frechet1}
\end{equation}
\textcolor{black}{From \eqref{eq:v-tilde-g}, we compute
\begin{align}
    \lambda(\Omega, s)^{2s} \norm{\tilde{v}[g]}_{L^2(\Omega)}^2 &\leq \sum_{k = 1}^{\infty} \lambda_k^{2s}|\langle\tilde{v}[g], \phi_k \rangle_{\Omega}|^2\\
    &= \norm{(-\Delta_{D, 0})^s \tilde{v}[g]}^2_{L^2(\Omega)}\\
    &\leq 2 \norm{q \tilde{v}[g]}^2_{L^2(\Omega)} + 2 \norm{q \mP^s g}^2_{L^2(\Omega)}\\
    &\leq 2 \norm{q}^2_{L^{\infty}(\Omega)} \norm{\tilde{v}[g]}^2_{L^2(\Omega)} + 2 \norm{q}_{L^{\infty}(\Omega)}^2\norm{\mP^{s}g}^2_{L^2(\Omega)}.
\end{align}
Therefore
\begin{align}
    \norm{\tilde{v}[g]}_{L^2(\Omega)} \leq \frac{4}{\lambda(\Omega, s)^{s}} \norm{q}_{L^{\infty}(\Omega)}\norm{\mP^{s}g}_{L^2(\Omega)},
\end{align}
}for all $q$ with $\norm{q}_{L^{\infty}(\Omega)} \le \frac{1}{2} \lambda(\Omega,s)^s$. Combining the above inequality with \Cref{lem:well-posedness}, one reaches 
\begin{equation}
\norm{\tilde{v}[g]}_{L^{2}(\Omega)} \le C(\Omega,s) \norm{q}_{L^{\infty}(\Omega)} \norm{g}_{H^{s-\frac{1}{2}}(\partial\Omega)}. \label{eq:Frechet2}
\end{equation}
Combining \eqref{eq:Frechet1} and \eqref{eq:Frechet2}, we conclude \eqref{eq:Frechet}. 
\end{proof}

\begin{remark}[Born approximation] \label{rem:Born-approximation}
In view of \Cref{lem:well-posedness}, we define the Green's operator $\mG^{s}:H^{-s}(\Omega)\rightarrow H^{s}(\Omega)$ by $\mG^{s} F:=u_{F}$, where $u_{F}$ is the unique solution of 
\begin{equation}
(-\Delta_{D,0})^{s} u_{F} = F \text{ in $\Omega$} ,\quad u_{F}|_{\partial\Omega}=0. \label{eq:sch-linearized2}
\end{equation}
Since $\mG^{s}:L^{2}(\Omega)\rightarrow L^{2}(\Omega)$ is a bounded linear operator, then if $\norm{q}_{L^{\infty}(\Omega)}$ is sufficiently small, then we have $\norm{\mG^{s}\circ\mM_{q}}_{L^{2}(\Omega)\rightarrow L^{2}(\Omega)}<1$, where $\mM_{q}$ is the multiplication operator by $q\in L^{\infty}(\Omega)$. In this case, one can verify that 
\begin{equation*}
\begin{aligned}
\tilde{v}[g] &= \sum_{j=1}^{\infty} \left( (\mG^{s}\circ\mM_{q})^{j}\circ\mP^{s} \right) g \quad \text{(converge in $L^{2}(\Omega)$),} \\
v[g] &= (\mG^{s}\circ\mM_{q}\circ\mP^{s})g.
\end{aligned}
\end{equation*}
One sees that $v[g]$ is exactly the principal term of $\tilde{v}[g]$, in other words, $v[g]$ is the Born approximation of $\tilde{v}[q]$. 
\end{remark}

\section{Proof of \texorpdfstring{\Cref{thm:main}}{Theorem~\ref{thm:main}}}

\begin{lemma}[Alessandrini identity] \label{lem:Alessandrini}
Let $\frac{1}{2}<s<1$ and given any $g,h\in C^{\infty}(\partial\Omega)$. We define $u_{g}=\mP^{s} g \equiv \mP g$ and $u_{h} = \mP^{s} h \equiv \mP h$ {\rm (}see \eqref{eq:Poisson-equivalence}{\rm )}. Then we have the identity 
\begin{equation}
\br{h,(\rmd\Lambda^{s}[q])g}_{\partial\Omega} = -\int_{\Omega} qu_{h}u_{g}\,\rmd x\quad\text{for all $g,h\in C^{\infty}(\partial\Omega)$.}\label{eq:alessandrini-identity}
\end{equation}
\end{lemma}

\begin{proof}
Choosing $u=u_{h}$ and $v=\tilde{u}_{g}=\mG^{s}(qu_{g})$ in the integration by parts formula in \Cref{lem:integration-by-parts}, we see that 
\[
\br{h,\partial_{\nu}w_{\tilde{u}_{g}}}_{\partial\Omega} = \left((-\Delta_{D})^{s}u_{h},\tilde{u}_{g}\right)_{L^{2}(\Omega)} - \left( u_{h},(-\Delta_{D,0})^{s}\tilde{u}_{g} \right)_{L^{2}(\Omega)} = -\int_{\Omega}qu_{h}u_{g}\,\rmd x,
\]
where $w_{v}$ is the function given in \Cref{lem:integration-by-parts}, which concludes \eqref{eq:alessandrini-identity}. 
\end{proof}

We are now ready to prove our main result modifying the ideas in \cite{SU87Calderon}. 

\begin{proof}[Proof of \Cref{thm:main}] 
Using \Cref{lem:Alessandrini}, we have 
\[
\br{h,(\rmd \Lambda^{s}[q^{(j)}])g}_{\partial\Omega} = -\int_{\Omega} q^{(j)}u_{h}u_{g} \, \rmd x \quad \text{for all $j=1,2$.}
\]
Therefore from \eqref{eq:DN-map-data} we know that 
\begin{equation}
\int_{\Omega} \left(q^{(1)}-q^{(2)}\right)u_{h}u_{g}\,\rmd x=0.\label{eq:Alessandrini-result}
\end{equation}
For each $\xi\in\mR^{n}\setminus\{0\}$, choose  $\eta \in \mR^n\setminus\{0\} $ such that $\eta \cdot \xi=0$, and $ |\eta|=|\xi|$. Then the functions of the form $e^{\alpha \cdot x}$ are solution of $-\Delta (\cdot)=0$  if $\alpha \cdot \alpha=0$ for any $\alpha\in \mathbb{C}^n$. We next choose $ u_h= e^{\alpha \cdot x}$ and $u_g= e^{\beta \cdot x}$, where $ \alpha = \frac{1}{2}(\eta-\bfi \xi)$ and $\beta= \frac{1}{2}(- \eta-\bfi\xi)$. 
Since $u_{h}u_{g} = e^{-\bfi\xi\cdot x}$, then we have\footnote{We also can obtain \eqref{eq:Fourier-difference} from \eqref{eq:Alessandrini-result} using the complex geometrical optics solutions as in \cite{Gun19CalderonNote,SU87Calderon}.}
\begin{equation}
\left(\chi_{\Omega}(q^{(1)}-q^{(2)})\right)\ehat(\xi) = \int_{\Omega} e^{-\bfi\xi\cdot x}\left(q^{(1)}-q^{(2)}\right)\,\rmd x = 0 \quad \text{for all $\xi\in\mR^{n}\setminus \{0\}$,}\label{eq:Fourier-difference}
\end{equation}
By Paley-Wiener theorem, we conclude that $q_1=q_2$ in $\Omega$.  
	
We now prove the stability result by modifying the ideas in \cite{Ale88stable}. Without loss of generality, we may assume
that 
\begin{equation}
\norm{\rmd\Lambda^{s}[q^{(1)}]-\rmd\Lambda^{s}[q^{(2)}]}_{*} \neq 0. \label{eq:WLOG}
\end{equation}
Suppose $\Omega \subset B(0, R)$. Plugging $u_g$ and $u_h$ into \eqref{eq:alessandrini-identity}, we obtain\begin{equation*}
\begin{aligned}
& \abs{\left(\chi_{\Omega}(q^{(1)}-q^{(2)})\right)\ehat(\xi)} \\
& \quad \le  \norm{\rmd\Lambda^{s}[q^{(1)}] - \rmd\Lambda^{s}[q^{(2)}]}_{*} \norm{u_{h}}_{H^{\frac{1}{2}}(\partial\Omega)} \norm{u_{g}}_{H^{\frac{1}{2}}(\partial\Omega)}\\
& \quad \le C \abs{\xi}^2 e^{{\abs{\xi}R}}\norm{\rmd \Lambda^{s}[q^{(1)}] - \rmd\Lambda^{s}[q^{(2)}]}_{*}
\end{aligned}
\end{equation*}
where the second inequality follows from $\norm{e^{\alpha\cdot x}}_{H^{1}(\Omega)}\le C(\Omega)\abs{\xi} e^{\frac{\abs{\xi}R}{2}},\, \norm{e^{\beta\cdot x}}_{H^{1}(\Omega)}\le C(\Omega)\abs{\xi} e^{\frac{\abs{\xi}R}{2}} $ and the boundedness of the Dirichlet trace operator ${\rm Tr}:H^{1}(\Omega)\rightarrow H^{\frac{1}{2}}(\partial\Omega)$. Let $\rho>0$ be a constant to be determine later, we see that 
\begin{equation}
\begin{aligned}
& \norm{\chi_{\Omega}(q^{(1)}-q^{(2)})}_{H^{-1}(\mathbb{R}^{n})}^2 \\
& \quad = \left(\int_{\abs{\xi}\le\rho} + \int_{\abs{\xi}>\rho}\right)\frac{\abs{\left(\chi_{\Omega}(q^{(1)}-q^{(2)})\right)\ehat(\xi)}^{2}}{1+\abs{\xi}^{2}} \, \rmd \xi \\
& \quad \le C\rho^{4} e^{2\rho R}\norm{\rmd\Lambda^{s}[q^{(1)}] - \rmd\Lambda^{s}[q^{(2)}]}_{*}^2 + \frac{1}{1+\rho^{2}}\norm{q^{(1)}-q^{(2)}}_{L^{2}(\Omega)}^{2} \\
& \quad \le C\left( e^{6\rho R}\norm{\rmd\Lambda^{s}[q^{(1)}]-\rmd\Lambda^{s}[q^{(2)}]}_{*}^2 + \frac{1}{\rho^{2}}\right) \quad \text{(using \eqref{eq:apriori})}
\end{aligned}\label{eq:Fourier-difference2}
\end{equation}
We now choose
\[
\rho = \frac{\abs{\operatorname{log}(\norm{\rmd\Lambda^{s}[q^{(1)}] - \rmd\Lambda^{s}[q^{(2)}]}_{*})}}{6R}\quad\text{(which is valid by \eqref{eq:WLOG})},
\]
Define, a modulus of continuity
\[w(t) := \sqrt{C}\left( e^{\abs{\operatorname{log}t}}t^2 + \frac{36 R^2}{\abs{\operatorname{log}t}^2}\right)^{\frac{1}{2}} \quad \text{for } t > 0.  \]
There exists a constant $\tilde{C} > 0$ such that
\[w(t) \leq \frac{\tilde{C}}{\abs{\operatorname{log}t}}, \quad 0 < t < 1/e.\]

Hence \eqref{eq:Fourier-difference2} implies 
\[
\begin{aligned}
& \norm{\chi_{\Omega}(q^{(1)}-q^{(2)})}_{H^{-1}(\mR^{n})} \le w(\norm{\rmd\Lambda^{s}[q^{(1)}] - \rmd\Lambda^{s}[q^{(2)}]}_{*}).
\end{aligned}
\]
which is our desired result. 
\end{proof}

\section{Proof of \texorpdfstring{\Cref{th_desnity}}{Theorem~\ref{th_density}}} 

Before proving \Cref{th_desnity}, we first present several auxiliary lemmas. We begin with the following lemma, which can be proved by adapting the arguments from \cite[Step~1 in the proof of Theorem~2.3]{SS_linearized} and \cite[Lemma 5.1]{FIKO_jst}.

\begin{lemma}\label{seq-of-fun}
Let $D\subset\mR^{n-1}$ be an open set, fix $M>0$ and set $\tilde{D}:=(-M, M)\times D$. 
Consider the symmetric matrix 
\begin{equation*}
\theta^{2}(y_1,y') = \left(\begin{array}{cccc} 
\theta_{11}^{2}(y_{1},y') & \theta_{12}^{2}(y_{1},y') & \cdots & \theta_{1n}^{2}(y_{1},y') \\ 
\theta_{12}^{2}(y_{1},y') & \theta_{22}^{2}(y_{1},y') & \cdots & \theta_{2n}^{2}(y_{1},y') \\ 
\vdots & \vdots & \ddots & \vdots \\ 
\theta_{1n}^{2}(y_{1},y') & \theta_{2n}^{2}(y_{1},y') & \cdots & \theta_{nn}^{2}(y_{1},y')
\end{array}\right) \quad \text{for all $(y_1,y')\in\tilde{D}$}
\end{equation*}
whose entries are smooth, bounded, and compactly supported in $\tilde{D}$. 
Define  
\begin{equation*}
\phi^{2} = \left(\begin{array}{cccc} 
\phi_{11}^{2} & \phi_{12}^{2} & \cdots & \phi_{1n}^{2} \\ 
\phi_{12}^{2} & \phi_{22}^{2} & \cdots & \phi_{2n}^{2} \\ 
\vdots & \vdots & \ddots & \vdots \\ 
\phi_{1n}^{2} & \phi_{2n}^{2} & \cdots & \phi_{nn}^{2}
\end{array}\right) := \theta_{11}\id_{n} - \theta^{2}. 
\end{equation*} 
Fix a unit vector $\eta\in\mR^{n}$ orthogonal to the first coordinate vector $e_{1}\in\mR^{n}$. 
For each $x'\in D$, write $x'=(x_2,x'')$, where $x_2$ is parallel to $\eta$ and each component of $x''$ is orthogonal to it. Let $\widehat{(\cdot)}$ denote the partial Fourier transform with respect to the $x_1$ variable. Suppose that 
\begin{align}\label{0.13}
\int_{\mathbb{R}}\left(\hat{\phi}^2(\lambda, y'):(\eta\otimes \eta) + 2 \bfi \sum_{j=1}^{n}\hat{\theta}_{1j}^2(\lambda, y') \eta_j\right)y_2e^{\lambda y_2} \,\rmd y_2 = 0 
\end{align}
for all $\lambda \in \mathbb{R}$. 
Then there exists a sequence $\{\varphi_k\}_{k = 0}^{\infty} \subset C^{\infty}(\overline{D})$ with $\varphi_k = \partial_{\nu}\varphi_k = 0$ on $\p D$ such that for every $k \geq 0$,
\begin{subequations}\label{eq:0.15-0.16}
\begin{align}\label{0.15}
\partial_{\lambda}^k\hat{\phi}_{ij}^{2}(0, y') &= \partial_{i}\partial_{j}\varphi_k(y') + k(k - 1)\delta_{ij}\varphi_{k - 2}(y') \quad \text{and}\\
\partial_{\lambda}^k \hat{\theta}_{1j}^2(0, y') &= - \bfi k \partial_j \varphi_{k - 1}(y')  \label{0.16}
\end{align}
\end{subequations}
for all $i,j\in\{2,\cdots,n\}$. Here we adopt the convention $\varphi_{-2} = \varphi_{-1} = 0$. 
\end{lemma}

\begin{remark}\label{rem:Paley-Wiener}
Note that since both $\theta^2$ and $\phi$ are compactly supported in the first variable, their Fourier transforms $\hat{\theta}(\lambda, y')$ and $\hat{\phi}(\lambda, y')$ are analytic in $\lambda$ by the Paley-Wiener theorem, see, e.g., \cite[Theorem~10.2.1(i)]{FJ98Distribution}. 
\end{remark}

\begin{proof}[Proof of \Cref{seq-of-fun}]
First, set $\lambda = 0$ in \eqref{0.13}, the replace $\eta$ with $-\eta$ and set $\lambda = 0$ again. This yields the two equations 
\begin{equation*} 
\int_{\mR}\hat{\phi}^2(0, y'): (\eta\otimes\eta) y_2 \, \rmd y_2 = 0 
\end{equation*} 
and
\begin{equation*} 
\int_{\mR} \sum_{j=1}^{n}\hat{\theta}_{1j}^2(0, y') \eta_j y_2 \mathrm{d}y_2 = 0.
\end{equation*} 
By \cite[Theorem 2.17.2]{Sharafutdinov_book}, there exists $\varphi_0 \in C^{\infty}(\overline{D})$, with $\varphi_0 = \partial_{\nu} \varphi_0 = 0$ on $\partial D$, such that
\begin{equation*}
\begin{aligned}
\hat{\theta}_{1j}^{2}(0, y') = 0 \quad \text{and}\\
\hat{\phi}_{ij}^2(0, y') = \partial^2_{ij} \varphi_0(y'), 
\end{aligned}
\end{equation*}
for all $i, j \in \{2, \dots, n\}$. 

We proceed by induction on $k$. Assume \eqref{eq:0.15-0.16} holds for all $k < k_0$. Differentiating \eqref{0.13} $k_0$ times with respect to $\lambda$ gives 
\begin{equation*}
\begin{aligned}
0=&   \int_{\mR} \binom{k_0}{0}\left(\partial_{\lambda}^{k_0}\hat{\phi}^2(\lambda, y'): (\eta\otimes\eta) + 2 \bfi  \partial_{\lambda}^{k_0} \sum_{j=1}^{n} 
\hat{\theta}_{1j}^2(\lambda, y') \eta_j\right)y_2 e^{\lambda y_2} \\
&+ \binom{k_0}{1}\left(\partial_{\lambda}^{k_0 - 1}\hat{\phi}^2(\lambda, y') : (\eta\otimes\eta) + 2 \bfi \sum_{j=1}^{n} \partial_{\lambda}^{k_0 - 1} \hat{\theta}_{1j}^2(\lambda, y') \eta_j\right)y_2^2 e^{\lambda y_2} \\
&+ \dots + \binom{k_0}{k_0}\left(\hat{\phi}^2(\lambda, y') : (\eta\otimes\eta) + 2 \bfi \sum_{j=1}^{n} \hat{\theta}_{1j}^2(\lambda, y') \eta_j\right)y_2^{k_0 + 1}e^{\lambda y_2} \, \rmd y_2.
\end{aligned}
\end{equation*}
Setting $\lambda = 0$ and replacing $\eta$ by $-\eta$ as before, the induction hypothesis gives 
\begin{equation*} 
\int_{\mR} y_2 \left(\partial_{\lambda}^{k_0}\hat{\phi}^2(0, y') -  k_0(k_0 - 1) \varphi_{k_0 - 2}(y') \id\right) : (\eta\otimes\eta) \, \rmd y_2 = 0 
\end{equation*} 
and
\begin{equation*}
\int_{\mR} y_2 \sum_{j=1}^{n} \left(\partial_{\lambda}^{k_0}\hat{\theta}_{1j}^2(0, y') + \bfi k_0 \partial_j \phi_{k_0 - 1}(y') \right)\eta_{j} \, \rmd y_2 = 0.
\end{equation*}
Applying \cite[Theorem 2.17.2]{Sharafutdinov_book} again, there exists $\varphi_{k_0} \in C^{\infty}(\overline{D})$ with $\varphi_{k_0} = \partial_{\nu} \varphi_{k_0} = 0$ on $\partial D$ such that
\begin{equation*} 
\begin{aligned}
\partial_{\lambda}^{k_0}\hat{\phi}_{ij}^2(0, y') &= \partial_{ij}^2\varphi_{k_0}(y') + k_0(k_0 - 1)\delta_{ij}\varphi_{k_0 - 2}(y') \quad \text{and}\\
\partial_{\lambda}^{k_0} \hat{\theta}_{1j}^2(0, y') &= - \bfi k_0 \partial_j \varphi_{k_{0} - 1}(y').
\end{aligned}
\end{equation*} 
for all $i, j \in \{2, \dots, n\}$.  
This completes the proof by induction. 
\end{proof}

Adapting the approach in \cite[Lemma~2.6]{BKS_20} and \cite{FIKO_jst}, we now prove the following lemma.

\begin{lemma}\label{lem:gauge-lemma}
Let $\theta^{2}$ be the symmetric matrix given in \Cref{seq-of-fun}, with $D=(0,L)\times\cdots\times(0,L)\subset \mathbb{R}^{n-1}$ for some $L>0$. For each unit vector $\eta\in\mR^{n}$ orthogonal to the first coordinate vector $e_1\in\mR^{n}$, define the transport operator 
\begin{equation*}
T_{\eta} := 2(e_1 + \bfi\eta)\cdot\nabla. 
\end{equation*}
Assume that for all such $\eta$,  
\begin{equation*}
0= \int_{\tilde{D}} \left( \theta^2 : (e_1+\bfi\eta)\otimes(e_1+\bfi\eta) \right) a_0  b_0,  
\end{equation*}
for all $a_0,b_0\in C^{\infty}(\tilde{D})$ satisfying $T_{\eta}^{2}a_{0}=0$ and $T_{\eta}b_{0}=0$. 
Then there exist scalar functions $\psi,w\in C^{\infty}(\tilde{D})$, with $\psi = \partial_{\nu}\psi = 0$ on $\partial\tilde{D}$, such that
\begin{equation*}
\theta^2 = \nabla^{\otimes 2}\psi + w\id_{n}.
\end{equation*} 
\end{lemma} 

\begin{remark*}
In particular, 
\begin{equation} 
\psi(y_1, y') := \sum_{j = 2}^{n}\int_{-\infty}^{y_1} \int_{0}^{1}y_j \theta_{1j}^2(s, ty') \mathrm{d}t \mathrm{d}s. \label{eq:smooth-psi}
\end{equation}
Moreover if $\theta^2\in C^k(\tilde{D})$, then one can obtain that $w, \psi\in C^k(\tilde{D})$.
\end{remark*}

\begin{proof}[Proof of \Cref{lem:gauge-lemma}]
Since $\eta\in\mR^{n}$ is a unit vector orthogonal to $e_1\in\mR^{n}$, we have 
\begin{equation}\label{0.2}
0= \int_{\tilde{D}} \left(\theta^2_{11} + 2 \bfi \sum_{j = 2}^{n}\theta^2_{1j}\eta_j - \sum_{i, j = 2}^{n}\theta^2_{ij}\eta_i \eta_j\right) a_{0} b_{0}. 
\end{equation}
Define, for $y \in \mathbb{R}^n$ and $\eta' \in \mathbb{R}^{n - 1}$, 
\begin{equation*}
F(y, \eta') = \theta^2_{11}(y) + 2 \bfi \sum_{j = 2}^{n}\theta^2_{1j}(y)\eta'_j - \sum_{i= 2}^{n}\sum_{j = 2}^{n}\theta^2_{ij}(y)\eta_i' \eta_j'.
\end{equation*}
Choose $a_0(y) = y_2 g(y'')e^{- \bfi \lambda (y_1 + \bfi y_2)}$ for all $y=(y_1,y_2,y'')\in\tilde{D}$ and $b_0\equiv 1$, where $g$ is an arbitrary smooth function. Substituting into \eqref{0.2} yields 
\begin{equation*}
0= \int_{\mathbb{R}^{n - 2}}\left(\int_{\mR^2} F(y_1, y_2, y'', \eta')y_2 e^{- \bfi \lambda (y_1 + \bfi y_2)} \, \rmd y_1 \, \rmd y_2 \right) g(y'') \, \rmd y''.
\end{equation*}
By the arbitrariness of $g$, we obtain $0 = \int_{\mR^2} F(y, \eta)y_2 e^{- \bfi \lambda (y_1 + \bfi y_2)} \, \rmd y_1 \, \rmd y_2$, that is, 
\begin{equation*}
0 = \int_{\mathbb{R}} \hat{F}(\lambda, y', \eta) y_2 e^{\lambda y_2} \, \rmd y_2.
\end{equation*}
where $\widehat{(\cdot)}$ denotes the partial Fourier transform in the $y_1$ variable. Equivalently, 
\begin{equation*}
0 = \int_{\mR}\left(\hat{\phi}^2(\lambda, y'):(\eta\otimes\eta) + 2 \bfi \sum_{j=1}^{n} \hat{\theta}_{1j}^2(\lambda, y') \eta_j\right)y_2e^{\lambda y_2} \,\rmd y_2,
\end{equation*}
with $\phi^2 = \theta_{11}\id - \theta^2$. 

Next, define $\psi$ as in \eqref{eq:smooth-psi}. 
By \eqref{0.16}, we have
\begin{equation*}
\psi(y_1, y')= \sum_{j = 2}^{n} \int_{0}^{1} y_j \hat{\theta^2_{1j}}(0, ty') \, \mathrm{d}t = 0 \quad \text{for all sufficiently large positive $y_1$.}
\end{equation*}
It also vanishes for large negative $y_1$ since $\supp\, \theta^2_{1j} \subset \tilde{D}$. 
For each $y'\in\overline{D}$, the function $\lambda\mapsto \hat{\psi}(\lambda,y')$ is analytic (cf. \Cref{rem:Paley-Wiener}), and it admits the expansion 
\begin{equation*}
\hat{\psi}(\lambda, y') = \sum_{k = 0}^{\infty} \frac{\tilde{\psi}_k(y')}{k!} \lambda^k.
\end{equation*}
From 
\begin{align}\label{first der of psi}
\partial_{y_1} \psi(y_1, y') = \sum_{j = 2}^{n} \int_{0}^{1} y_j \theta_{1j}^2(y_1, ty') \, \rmd t, 
\end{align}
we obtain, after taking the Fourier transform in $y_1$, 
\begin{equation*}
\bfi \lambda \hat{\psi}(\lambda, y') = \sum_{j = 2}^n \int_{0}^{1} y_j \hat{\theta}_{1j}^2(\lambda, ty') \, \rmd t.
\end{equation*}
Differentiating $(k + 1)$-times in $\lambda$ and evaluating at $\lambda = 0$, we deduce 
\begin{equation*} 
\begin{aligned}
& \bfi (k + 1) (\hat{\psi})^{(k)}(0, y') = \sum_{j = 2}^{n} \int_{0}^{1} y_j (\hat{\theta}_{1j}^2)^{(k + 1)}(0, ty') \, \rmd t\\
& \quad \overset{\eqref{0.16}}{=} -\bfi (k + 1)\sum_{j = 2}^{n} \int_{0}^{1} y_j \partial_j \phi_{k}(ty') \, \rmd t = -\bfi (k + 1) \int_{0}^{1} \frac{d}{dt}\phi_{k}(ty') \, \rmd t = - \bfi (k + 1) \phi_{k}(y'), 
\end{aligned}
\end{equation*}
and hence 
\begin{equation*}
\tilde{\psi}_k(y') = (\hat{\psi})^k(0, y') = - \phi_{k}(y') \quad \text{for all $k \geq 0$ and all $y' \in \overline{D}$.}
\end{equation*}
Therefore, 
\begin{equation*}
\hat{\psi}(\lambda, y') = -\sum_{k = 0}^{\infty} \frac{\phi_k(y')}{k!}\lambda^k. 
\end{equation*}
Using \eqref{0.15}, we compute 
\begin{equation*} 
\begin{aligned}
& \hat{\phi}_{ij}^2(\lambda, y') = \sum_{k = 0}^{\infty} \frac{\partial_{\lambda}^k \hat{\phi_{ij}^2}(0, y')}{k!} \lambda^k = \sum_{k = 0}^{\infty} \frac{\partial_{ij}^2 \phi_k(y')}{k!} \lambda^k + \sum_{k = 0}^{\infty} \frac{\delta_{ij} \phi_{k - 2}(y')}{(k - 2)!}\lambda^k\\
& \quad = - \partial^2_{ij}\hat{\psi}(\lambda, y') - \lambda^2 \delta_{ij} \hat{\psi}(\lambda, y') \quad \text{for all $i, j \in \{2, \dots, n\}$}. 
\end{aligned}
\end{equation*} 
Taking inverse Fourier transforms yields 
\begin{equation}\label{sym}
\theta^2_{ij} = \partial_{ij}^2 \psi + \delta_{ij}(\theta^2_{11} - \partial_{11}^2 \psi) \quad \text{for all $i, j \in \{2, \dots, n\}$}.
\end{equation}
Moreover, since $\partial_{\lambda}^{k_0}{(\partial_{\ell} \hat{\theta}_{1j}^2 - \partial_{j} \hat{\theta}_{1\ell}^2)}(0, y') = 0$ for all $k_0\ge 0$, from \eqref{0.16} we have 
\begin{align}\label{nec. cond.}
\partial_{\ell} \theta^2_{1j} = \partial_{j} \theta^2_{1\ell} \quad \text{for all $j, \ell \in \{2, \dots, n\}$.}
\end{align}
Combining this with \eqref{first der of psi} and \eqref{nec. cond.}, we obtain 
\begin{equation*}
 \theta^2_{1j} = \partial^2_{1j} \psi + \delta_{1j}(\theta^2_{11} - \partial_{11}^2 \psi) \quad \text{for all $j=2,\cdots,n$}. 
\end{equation*}
Together with \eqref{sym}, this completes the proof. 
\end{proof}

We now choose an orthonormal frame 
\begin{equation*}
\left\{ \eta_{1}=e_{1} , \eta_{2} ,\cdots ,\eta_{n} \right\}, 
\end{equation*}
and associated transport operator 
\begin{equation*}
T:=2(e_{1}+\bfi\eta_{2})\cdot\nabla 
\end{equation*}
where $\{\eta_{2},\cdots,\eta_{n}\}$ are unit vectors lying in the hyperplane perpendicular to $e_{1}$. Before proving \Cref{th_desnity}, we require the construction of special solutions, as stated in the following lemma: 

\begin{lemma}\label{lem:CGO-linear-phase} 
We define the associated transport operator $T:=2(e_{1}+\bfi\eta_{2})\cdot\nabla$ and let $h>0$ be a sufficiently small parameter. 
\begin{enumerate}
\renewcommand{\labelenumi}{\theenumi}
\renewcommand{\theenumi}{\rm (\alph{enumi})} 
\item\label{itm:CGO-a} For each $m\in\mN$, let $a_{0},\cdots,a_{m-1}\in C^{\infty}(\overline{\Omega})$ solve 
\begin{equation*}
Ta_{j} = \Delta a_{j-1} \text{ in $\Omega$ for all $j=0,\cdots,m-1$} 
\end{equation*}
with the convention $a_{-1}\equiv 0$. Then there exists a remainder term $r(\cdot;h)\in H_{[h]}^{2}(\Omega)$ such that 
\begin{equation*}
w(x,h)=e^{\frac{-x\cdot(e_{1}+\bfi\eta_{2})}{h}}(A_{m}(x;h)+r(x;h)) \quad \text{with} \quad A_{m}(x;h) = \sum_{j=0}^{m-1}h^{j}a_{j}(x) 
\end{equation*}
is harmonic and satisfies 
\begin{equation*}
\norm{r(\cdot;h)}_{H_{[h]}^{2}(\Omega)} \le Ch^{m}
\end{equation*}
for some positive constant $C$ independent of $h$ and $m$, but depends on $\Omega$, $\eta_{2}$ and $a_{m-1}$. 
\item\label{itm:CGO-b} For each $m\in\mN$, let $b_{0},\cdots,b_{m-1}\in C^{\infty}(\overline{\Omega})$ solve 
\begin{equation*}
T^{2}b_{j} = -2\Delta T b_{j-1} - \Delta^{2}b_{j-2} \text{ in $\Omega$ for all $j=0,\cdots,m-1$} 
\end{equation*}
with the convention $b_{-2}\equiv b_{-1}\equiv 0$. Then there exists a remainder term $\tilde{r}(\cdot;h)\in H_{[h]}^{4}(\Omega)$ such that 
\begin{equation*}
\tilde{w}(x,h)=e^{\frac{x\cdot(e_{1}+\bfi\eta_{2})}{h}}(B_{m}(x;h)+\tilde{r}(x;h)) \quad \text{with} \quad B_{m}(x;h) = \sum_{j=0}^{m-1}h^{j}b_{j}(x) 
\end{equation*}
is biharmonic and satisfies 
\begin{equation*}
\norm{\tilde{r}(\cdot;h)}_{H_{[h]}^{4}(\Omega)} \le Ch^{m}
\end{equation*}
for some positive constant $C$ independent of $h$ and $m$, but depends on $\Omega$, $\eta_{2}$, $b_{m-2}$ and $b_{m-1}$.  
\end{enumerate}
\end{lemma}

We postpone the proof of \Cref{lem:CGO-linear-phase} to \Cref{sec:CGO-proof}. Indeed, \Cref{lem:CGO-linear-phase} is a special case of \Cref{prop:CGO-general}, corresponding to the choice $(\varphi,\psi)=(-e_{1}\cdot x , -\eta_{2}\cdot x)$ for \Cref{lem:CGO-linear-phase}\ref{itm:CGO-a} and $(\varphi,\psi)=(e_{1}\cdot x , \eta_{2}\cdot x)$ for \Cref{lem:CGO-linear-phase}\ref{itm:CGO-b}. With \Cref{lem:CGO-linear-phase} at hand, we can now prove \Cref{th_desnity} for $n\ge 3$ using the strategy described in \cite[Remark~5.4]{SS_linearized}. 

\begin{proof}[Proof of \Cref{th_desnity} for $n\ge 3$]
Substituting $u=\tilde{w}$ and $v=w$, where $w$ and $\tilde{w}$ are complex geometric optics (CGO) solutions from \Cref{lem:CGO-linear-phase} with $m=4$, into \eqref{eq:integral-identity} yields 
\begin{equation}
\begin{aligned}
0 &= \int_{\Omega} \left[ \theta^{2}:\nabla^{\otimes 2} \left(e^{\frac{x\cdot(e_{1}+\bfi\eta_{2})}{h}}(b_{0}(x)+b_{1}(x)h+b_{2}(x)h^{2}+b_{3}(x)h^{3} + \tilde{r}(x;h))\right) \right. \\ 
& \quad + \theta^{1}\cdot\nabla \left(e^{\frac{x\cdot(e_{1}+\bfi\eta_{2})}{h}}(b_{0}(x)+b_{1}(x)h+b_{2}(x)h^{2}+b_{3}(x)h^{3} + \tilde{r}(x;h))\right) \\ 
& \quad \left. + \theta^{0}e^{\frac{x\cdot(e_{1}+\bfi\eta_{2})}{h}}(b_{0}(x)+b_{1}(x)h+b_{2}(x)h^{2}+b_{3}(x)h^{3} + \tilde{r}(x;h)) \right] \times \\ 
& \qquad \times e^{\frac{-x\cdot(e_{1}+\bfi\eta_{2})}{h}}(a_{0}(x)+a_{1}(x)h+a_{2}(x)h^{2}+a_{3}(x)h^{3} + r(x;h)) \, \rmd x 
\end{aligned} \label{eq:n3-eqn1}
\end{equation}
From now on, we divide the proof into three steps, based on the different powers of $h$.  

\medskip 

\noindent \emph{We begin with the $O(h^{-2})$ term.} To this end, we multiply \eqref{eq:n3-eqn1} by $h^{2}$ and then take the limit $h\rightarrow 0_{+}$ to conclude 
\begin{equation*}
0 = \int_{\Omega} \theta^{2}:((e_{1}+\bfi\eta_{2})\otimes(e_{1}+\bfi\eta_{2})) a_{0}b_{0} 
\end{equation*}
 We now utilize the \Cref{lem:gauge-lemma} and assume that  W.L.O.G  $\Omega\subset \tilde{D}$. This can be achieved via a translation, since $\Omega$ is bounded.
Using \Cref{lem:gauge-lemma}, there exist scalar functions $\psi,w\in C^{\infty}(\tilde{D})$, with $\psi = \partial_{\nu}\psi = 0$ on $\partial\tilde{D}$, such that
\begin{equation}
\theta^2 = \nabla^{\otimes 2}\psi + w\id_{n}. \label{2-tensor-gauge} 
\end{equation} 

\medskip 

\noindent \emph{We now turn to the $O(h^{-1})$ term.} To this end, we multiply \eqref{eq:n3-eqn1} by $h$ and then take the limit $h\rightarrow 0_{+}$ to conclude 
\begin{equation*}
\begin{aligned}
0 &= \int_{\Omega} \theta^{2}:((e_{1}+\bfi\eta_{2})\otimes(e_{1}+\bfi\eta_{2})) (a_{0}b_{1}+a_{1}b_{0}) \\
&\quad + 2\theta^{2}:((e_{1}+\bfi\eta_{2})\otimes\nabla a_0) b_0 + \theta^{1}\cdot(e_{1}+\bfi\eta_{2}) a_{0}b_{0}, 
\end{aligned}
\end{equation*}
provided $T^{2}a_{0}=0$, $T^{2}a_{1} = -T\Delta a_{0}$, $Tb_{0}=0$ and $Tb_{1}=\Delta b_{0}/2$. Note that the above integral identity is actually over $\tilde{D}$, as the coefficients are supported in $\overline{\Omega}$.  By \eqref{2-tensor-gauge}, we obtain 
\begin{equation*}
\begin{aligned}
0 &= \int_{\tilde{D}} (\nabla^{\otimes 2}\psi) : ((e_{1}+\bfi\eta_{2})\otimes(e_{1}+\bfi\eta_{2}))(a_{0}b_{1}+a_{1}b_{0}) + \int_{\tilde{D}} \theta^{1}\cdot(e_{1}+\bfi\eta_{2})a_{0}b_{0} \\ 
&\quad + 2\int_{\tilde{D}} (\nabla^{\otimes 2}\psi):((e_{1}+\bfi\eta_{2})\otimes\nabla a_{0})b_{0} + 2\int_{\tilde{D}} wTa_{0}b_{0}. 
\end{aligned}
\end{equation*}
By integrating by parts, we obtain 
\begin{equation}
\begin{aligned}
0 &= \int_{\tilde{D}} \psi(2Ta_{0}Tb_{1} + a_{0}\overbrace{T^{2}b_{1}}^{=\,0} + T^{2}a_{1}b_{0}) + \int_{\tilde{D}} \theta^{1}\cdot(e_{1}+\bfi\eta_{2})a_{0}b_{0} \\ 
&\quad + \int_{\tilde{D}} (-2) \nabla\psi\cdot\nabla Ta_{0}b_{0} + 2w Ta_{0}b_{0}. 
\end{aligned} \label{eq:gauge-proof1}
\end{equation}
Next, we choose $a_0$ such that $T a_0 = 0$, in which case the above expression simplifies to 
\begin{equation*}
0 = \int_{\tilde{D}} \theta^{1}\cdot(e_{1}+\bfi\eta_{2})a_{0}b_{0} 
\end{equation*}
because $T^{2}a_{1} = -T\Delta a_{0}=0$. Using the arguments used in  \cite{SS_linearized,Krupchy_Uhl_CMP}, we can show that 
\begin{equation}
\theta^{1} = \nabla\varphi \label{1-tensor-gauge} 
\end{equation}
for some smooth function $\varphi$ with $\varphi|_{\partial \tilde{D}}=0$. Substituting this into \eqref{eq:gauge-proof1} yields 
\begin{equation*} 
0 = \int_{\tilde{D}} \psi(2Ta_{0}Tb_{1} + T^{2}a_{1}b_{0}) + (2w-\varphi) Ta_{0}b_{0} -2 \int_{\tilde{D}} \nabla\psi\cdot\nabla Ta_{0}b_{0} . 
\end{equation*}
An integration by parts yields 
\begin{equation*}
0 = \int_{\tilde{D}} \psi(Ta_{0}\Delta b_{0} - \Delta Ta_{0}b_{0}) + \int_{\tilde{D}}(2w-\varphi)Ta_{0}b_{0} + \int_{\tilde{D}} 2\psi\Delta Ta_{0}b_{0} + 2\psi \nabla Ta_{0}\cdot \nabla b_{0}. 
\end{equation*} 
Substituting $Ta_{0}=e^{-\bfi\lambda(y_{1}+\bfi y_{2})}$ for $\lambda\neq 0$ and $b_{0}=g(y'')$, with $g$ an arbitrary smooth function, into the above equation yields 
\begin{equation}
0=\int_{\tilde{D}}\psi \Delta g(y'') e^{-\bfi\lambda(y_{1}+\bfi y_{2})} + (2w-\varphi)g(y'')e^{-\bfi\lambda}(y_{1}+\bfi y_{2}). \label{eq:gauge-proof2}
\end{equation}

\medskip 

\noindent \emph{Finally, we now turn to the $O(h^{-1})$ term.} To this end, we pass to the limit $h\rightarrow 0$ in \eqref{eq:n3-eqn1}, which yields 
\begin{equation*}
\begin{aligned}
0 &= \int_{\Omega} \theta^{2}:((e_{1}+\bfi\eta_{2})\otimes(e_{1}+\bfi\eta_{2}))(a_{0}b_{2}+a_{1}b_{1}+a_{2}b_{1}) \\
&\quad + \int_{\Omega} 2\theta^{2}:((e_{1}+\bfi\eta_{2})\otimes(\nabla a_{0}b_{1} + \nabla a_{1}b_{0})) \\ 
&\quad + \int_{\Omega} (\theta^{2} : \nabla^{\otimes 2}a_{0}) b_{0} + \theta^{1}\cdot(e_{1}+\bfi\eta_{2})(a_{0}b_{1}+a_{1}b_{0}) + \theta^{1}\cdot\nabla a_{0} b_{0} + \theta^{0}a_{0}b_{0}. 
\end{aligned}
\end{equation*}
Note that the above integral identity is actually over $\tilde{D}$, as the coefficients are supported in $\overline{\Omega}$.
Plugging \eqref{2-tensor-gauge} and \eqref{1-tensor-gauge} into the above equation, we obtain 
\begin{equation}
\begin{aligned}
0 &= \int_{\tilde{D}}\psi(2Ta_{0}Tb_{2} + a_{0}T^{2}b_{2}) - 2\nabla\psi\cdot(\nabla a_{0} Tb_{1} + \nabla Ta_{0} b_{1} + \nabla Ta_{1} b_{0}) \\
&\quad + \int_{\tilde{D}} \nabla^{\otimes 2}\psi : \nabla^{\otimes 2}a_{0}b_{0} \\ 
&\quad + \int_{\tilde{D}} 2w(Ta_{0}b_{1} + Ta_{1}b_{0}) + w\Delta a_{0}b_{0} - \varphi(Ta_{0}b_{1}+a_{0}Tb_{1} + Ta_{1}b_{0}) \\ 
&\quad +\int_{\tilde{D}} (-\varphi\Delta a_{0}b_{0} - \varphi \nabla a_{0}\cdot\nabla b_{0} + \theta^{0}a_{0}b_{0}). 
\end{aligned} \label{eq:gauge-proof3}
\end{equation}
Substituting $a_{0}\equiv 0$, $Ta_{1}=1$ and $b_{0}=e^{-\bfi\lambda(y_{1}+\bfi y_{2})}g(y'')$, with $g$ an arbitrary smooth function, into the above equatio yields 
\begin{equation*}
0 = \int_{\tilde{D}} (2w-\varphi)e^{-\bfi\lambda(y_{1}+\bfi y_{2})}g(y''). 
\end{equation*}
Using results from \cite{DKSU,SS_linearized} we conclude that $ \varphi \equiv 2w$. Now, \eqref{1-tensor-gauge} becomes 
\begin{equation*}
\theta^{1} = 2\nabla w 
\end{equation*}
and \eqref{eq:gauge-proof2} reduces to 
\begin{equation*}
0=\int_{\tilde{D}}\psi \Delta g(y'') e^{-\bfi\lambda(y_{1}+\bfi y_{2})}. 
\end{equation*}
By arbitrariness of $g$, we conclude that $\psi \equiv 0$, and \eqref{2-tensor-gauge} reduces to   
\begin{equation*}
\theta^2 = w\id_{n}.  
\end{equation*} 
At this stage, \eqref{eq:gauge-proof3} simplifies to 
\begin{equation*}
\begin{aligned}
0 &= \int_{\tilde{D}} -w\Delta a_{0}b_{0} - 2wa_{0}Tb_{1}  - 2w \nabla a_{0}\cdot\nabla b_{0} + \theta^{0}a_{0}b_{0} \\ 
&= \int_{\tilde{D}} -w\Delta a_{0}b_{0} - wa_{0}\Delta b_{0} - 2w \nabla a_{0}\cdot\nabla b_{0} + \theta^{0}a_{0}b_{0} 
\end{aligned} 
\end{equation*}
since $2Tb_{1}=\Delta b_{0}$. If we choose $Ta_{0}=0$ and $Tb_{0}=0$, then an integration by parts yields 
\begin{equation*}
0 = \int_{\tilde{D}}(\theta^{0}-\Delta w)a_{0}b_{0}=0. 
\end{equation*}
Choosing $a_{0}=g(y'')$ and $b_{0}=e^{-\bfi\lambda(y_{1}+\bfi y_{2})}$ for any $\lambda\neq 0$, with $g$ an arbitrary smooth function, we conclude that $\theta^{0}=\Delta w$, thereby completing the proof of the theorem. 
\end{proof}

Before proving \Cref{th_desnity} for $n=2$, we first recall the following stationary phase result. 
  
\begin{lemma}[{\cite[Proposition 2.3]{Grigis_book}}]\label{Stationary phase} Let $a \in C_c^{2N + 3}(\mathbb{R}^2)$ and let $A$ be a real, non-singular, and symmetric matrix. Then 
\begin{equation*}
\int_{\mR^2} e^{\frac{\bfi}{2h}y\cdot Ay}a(y)dy = 2 \pi h\frac{e^{\bfi\frac{\pi}{4}\operatorname{sgn}A}}{|\operatorname{det}A|^{1/2}}\left( \sum_{k = 0}^{N - 1} \frac{h^k}{k!} (P^ka)(0) + R_N(a, h) \right) \quad \text{as $h \to 0$},
\end{equation*}
where $P = \frac{1}{2\bfi} \langle D, A^{-1}D \rangle$, and
\begin{equation*}
|R_N(a, h)| \leq \frac{h^N}{N!} C_{A} \sum_{|\alpha| \leq 3} \norm{\partial^{\alpha} P^N a}_{L^1}.
\end{equation*}
Here, $\operatorname{sgn}{A}$ denotes for the signature of $A$, defined as the number of positive eigenvalues minus the number of negative eigenvalues. 
\end{lemma}

We now ready to prove \Cref{th_desnity} for $n=2$. 

\begin{proof}[Proof of \Cref{th_desnity} for $n=2$]
Throughout the proof, we simplify notation by identifying $z=x+\bfi y\cong (x,y)\in\mR^{2}$. Fix any $z_0\in\Omega$ and define $\phi(z)=(z-z_0)^{2}$. 

\medskip 

\noindent \emph{First, we choose $u(z) = {\overline{(z-z_0)}}e^{\frac{\phi}{h}}$ and $v(z) = e^{\frac{-\overline{\phi}}{h}}$ into \eqref{eq:integral-identity}.} 
Since 
\begin{equation*}
\partial_i u =  \left((e_1 - \bfi e_2)_i  + \frac{2(e_1 + \bfi e_2)_i |z-z_0|^2}{h}\right) e^{\frac{\phi}{h}}
\end{equation*}
and
\begin{equation*}
\partial^2_{ij}u = \left(\begin{aligned}
& \frac{4(e_1 - \bfi e_2)_i(e_1 + \bfi e_2)_j (z-z_0)}{h} \\
&\quad + \frac{2 (e_1 + \bfi e_2)_i(e_1 + \bfi e_2)_j \overline{(z-z_0)}}{h} \\ 
&\quad + \frac{4 (e_1 + \bfi e_2)_i(e_1 + \bfi e_2)_j (z-z_0)|z-z_0|^2}{h^2}
\end{aligned}\right)e^{\frac{\phi}{h}}
\end{equation*}
for all $i,j\in\{1,\cdots,n\}$, \eqref{eq:integral-identity} becomes 
\begin{equation*} 
\begin{aligned}
0 &= \int_{\mR^2} \sum_{i,j=1}^{2}\theta^2_{ij} e^{\frac{\phi - \bar{\phi}}{h}}\left(\begin{aligned}
& \frac{4(e_1 - \bfi e_2)_i(e_1 + \bfi e_2)_j (z-z_0)}{h^2} \\ 
&\quad + \frac{2 (e_1 + \bfi e_2)_i(e_1 + \bfi e_2)_j \overline{(z-z_0)}}{h^2} \\ 
&\quad + \frac{4 (e_1 + \bfi e_2)_i(e_1 + \bfi e_2)_j (z-z_0)|z-z_0|^2}{h^3}
\end{aligned} \right)\\
& \quad + \int_{\mR^2} \sum_{i=1}^{2}\theta_i^1 e^{\frac{\phi - \bar{\phi}}{h}}\left(\frac{(e_1 - \bfi e_2)_i}{h}  + \frac{2(e_1 + \bfi e_2)_i |z-z_0|^2}{h^2}\right) + \frac{\theta^{0} \overline{(z-z_0)}}{h}e^{\frac{\phi - \bar{\phi}}{h}}. 
\end{aligned} 
\end{equation*} 
Applying \Cref{Stationary phase}, we obtain 
\begin{equation*}
\begin{aligned}
0&= \sum_{i,j=1}^{2} 4(e_1 - \bfi e_2)_i(e_1 + \bfi e_2)_jP((z-z_0) \theta_{ij}^2)(z_0) \\
&\quad + \sum_{i,j=1}^{2} 2 (e_1 + \bfi e_2)_i(e_1 + \bfi e_2)_j P(\overline{(z-z_0)}\theta_{ij}^2)(z_0) \\
& \quad + \sum_{i,j=1}^{2} 4 (e_1 + \bfi e_2)_i(e_1 + \bfi e_2)_j\frac{P^2}{2}(\theta_{ij}^2 (z-z_0)|z-z_0|^2)(z_0) \\
&\quad + \sum_{i=1}^{2}(e_1 - \bfi e_2)_i \theta_i^1(z_0) + \sum_{i=1}^{2} 2(e_1 + \bfi e_2)_i P(|z-z_0|^2 \theta_i^1)(z_0) 
\end{aligned}
\end{equation*} 
where $P = -\frac{1}{4\bfi} \frac{\partial^2}{\partial x \partial y} = \frac{\bfi}{4}\frac{\partial^2}{\partial x \partial y}$.
Consequently, 
\begin{equation*} 
\begin{aligned}
0 &= \sum_{i,j=1}^{2}(e_1 - \bfi e_2)_i(e_1 + \bfi e_2)_j(-\partial_x \theta_{ij}^2 + \bfi \partial_y \theta_{ij}^2)(z_0) \\
&\quad + \sum_{i,j=1}^{2}2 (e_1 + \bfi e_2)_i(e_1 + \bfi e_2)_j \left(\frac{1}{4}(\partial_x \theta_{ij}^2 + \bfi \partial_y \theta_{ij}^2)(z_0)\right)\\
&\quad + \sum_{i,j=1}^{2}2 (e_1 + \bfi e_2)_i(e_1 + \bfi e_2)_j\left(\frac{-1}{4}(\partial_x \theta_{ij}^2 + \bfi \partial_y \theta_{ij}^2)(z_0)\right) \\
&\quad + \sum_{i=1}^{2}(e_1 - \bfi e_2)_i \theta_i^1(z_0), 
\end{aligned}
\end{equation*} 
which implies
\begin{equation*} 
\sum_{i,j=1}^{2}(e_1 - \bfi e_2)_i(e_1 + \bfi e_2)_j(-\partial_x \theta_{ij}^2 + \bfi \partial_y \theta_{ij}^2)(z_0) + \sum_{i=1}^{2}(e_1 - \bfi e_2)_i \theta_i^1(z_0) = 0.
\end{equation*}
Therefore,
\begin{equation}\label{2a}
-\partial_x (\theta_{11}^2 + \theta_{22}^2)(z_0) + \theta_1^1(z_0) + \bfi (\partial_y (\theta^2_{11} + \theta_{22}^2)(z_0) - \theta_2^1(z_0)) = 0 \quad \text{for all $z_0\in\Omega$.} 
\end{equation}

\medskip 

\noindent \emph{Next, we choose $u(z) = (z-z_0)e^{\frac{-\overline{\phi}}{h}}$ and $v(z) = e^{\frac{\phi}{h}}$ into \eqref{eq:integral-identity}.} 
Since 
\begin{equation*}
\partial_i u = \left((e_1 + \bfi e_2)_i  - \frac{2(e_1 - \bfi e_2)_i |z-z_0|^2}{h}\right)e^{\frac{-\bar{\phi}}{h}},
\end{equation*} 
and
\begin{equation*}
\partial^2_{ij}u = \left(\begin{aligned}
& -\frac{4(e_1 + \bfi e_2)_i(e_1 - \bfi e_2)_j}{h}\overline{(z-z_0)} \\ 
&\quad - \frac{2 (e_1 - \bfi e_2)_i(e_1 - \bfi e_2)_j}{h}(z-z_0) \\ 
&\quad + \frac{4 (e_1 - \bfi e_2)_i(e_1 - \bfi e_2)_j}{h^2}\overline{(z-z_0)}|z-z_0|^2
\end{aligned}\right)e^{\frac{-\bar{\phi}}{h}}.
\end{equation*} 
for all $i,j\in\{1,\cdots,n\}$, \eqref{eq:integral-identity} becomes 
\begin{equation*} 
\begin{aligned} 
0 &= \int_{\mathbb{R}^2} \sum_{i,j=1}^{2} \theta_{ij}^2 e^{\frac{\phi - \bar{\phi}}{h}}\left(\begin{aligned}
& -\frac{4(e_1 + \bfi e_2)_i(e_1 - \bfi e_2)_j}{h^2}\overline{(z-z_0)} \\ 
&\quad - \frac{2 (e_1 - \bfi e_2)_i(e_1 - \bfi e_2)_j}{h^2}(z-z_0) \\ 
&\quad + \frac{4 (e_1 - \bfi e_2)_i(e_1 - \bfi e_2)_j}{h^3}\overline{(z-z_0)}|z-z_0|^2
\end{aligned}\right)\\
& + \sum_{i=1}^{2} \theta_i^1 e^{\frac{\phi - \bar{\phi}}{h}}\left(\frac{(e_1 + \bfi e_2)_i}{h}  - \frac{2(e_1 - \bfi e_2)_i}{h^2}|z-z_0|^2\right) + \frac{\theta^0}{h}(z-z_0)e^{\frac{\phi - \bar{\phi}}{h}}. 
\end{aligned} 
\end{equation*} 
Applying \Cref{Stationary phase}, we obtain 
\begin{equation*} 
\begin{aligned}
0 &= -\sum_{i,j=1}^{2}4(e_1 + \bfi e_2)_i(e_1 - \bfi e_2)_jP(\overline{(z-z_0)}\theta_{ij}^2)(z_0) \\
&\quad - \sum_{i,j=1}^{2} 2 (e_1 - \bfi e_2)_i(e_1 - \bfi e_2)_j P((z-z_0)\theta_{ij}^2)(z_0) \\
& \quad + \sum_{i,j=1}^{2} 4 (e_1 - \bfi e_2)_i(e_1 - \bfi e_2)_j\frac{P^2}{2}(\theta_{ij}^2 \overline{(z-z_0)}|z-z_0|^2)(z_0) \\
&\quad + \sum_{i=1}^{2} (e_1 + \bfi e_2)_i \theta_i^1(z_0) - \sum_{i=1}^{2} 2(e_1 - \bfi e_2)_i P(|z-z_0|^2 a_i^1)(z_0) 
\end{aligned}
\end{equation*} 
where $P = -\frac{1}{4\bfi} \frac{\partial^2}{\partial x \partial y} = \frac{\bfi}{4}\frac{\partial^2}{\partial x \partial y}$. 
Hence
\begin{equation*} 
\begin{aligned}
0 &= -\sum_{i,j=1}^{2}(e_1 + \bfi e_2)_i(e_1 - \bfi e_2)_j(\partial_x \theta_{ij}^2 + \bfi \partial_y \theta_{ij}^2)(z_0) \\
&\quad - \sum_{i,j=1}^{2}2 (e_1 - \bfi e_2)_i(e_1 - \bfi e_2)_j \left(\frac{1}{4}(-\partial_x \theta_{ij}^2 + \bfi \partial_y \theta_{ij}^2)(z_0)\right)\\
&\quad + \sum_{i,j=1}^{2}2 (e_1 - \bfi e_2)_i(e_1 - \bfi e_2)_j\left(\frac{1}{4}(-\partial_x \theta_{ij}^2 + \bfi \partial_y \theta_{ij}^2)(z_0)\right) \\
&\quad + \sum_{i=1}^{2}(e_1 + \bfi e_2)_i \theta_i^1(z_0) 
\end{aligned} 
\end{equation*} 
which implies
\begin{equation*}
\sum_{i,j=1}^{2}(e_1 + \bfi e_2)_i(e_1 - \bfi e_2)_j(-\partial_x \theta_{ij}^2 - \bfi \partial_y \theta_{ij}^2)(z_0) + \sum_{i=1}^{2}(e_1 + \bfi e_2)_i \theta_i^1(z_0) = 0.
\end{equation*}
Therefore,
\begin{equation}\label{2b}
-\partial_x (\theta_{11}^2 + \theta_{22}^2)(z_0) + \theta_1^1(z_0) - \bfi (\partial_y (\theta_{11}^2 + \theta_{22}^2)(z_0) - \theta_2^1(z_0)) = 0 \quad \text{for all $z_0\in\Omega$}.
\end{equation}
After adding and subtracting \eqref{2a} and \eqref{2b}, we get
\begin{equation*}
\partial_{x}({\rm tr}\,(\theta^2))=\partial_x (\theta_{11}^2 + \theta_{22}^2) = \theta_1^1 ,\quad \partial_{y}({\rm tr}\,(\theta^2))=\partial_y (\theta_{11}^2 + \theta_{22}^2) = \theta_2^1 \quad \text{in $\Omega$}, 
\end{equation*}
that is, $\theta^{1} = \nabla({\rm tr}\,(\theta^2))=0$. 

\medskip 

\noindent \emph{Next, we choose $u(z) = e^{\frac{\phi}{h}}$ and $v(z) = \overline{(z-z_0)}e^{\frac{-\overline{\phi}}{h}}$ into \eqref{eq:integral-identity}.} Since 
\begin{equation*}
\partial_i u = \frac{2(e_1 + \bfi e_2)_i}{h}(z-z_0) e^{\frac{\phi}{h}}
\end{equation*}
and 
\begin{equation*}
\partial^2_{ij}u = \frac{2(e_1 + \bfi e_2)_{i}(e_1 + \bfi e_2)_{j}}{h} e^{\frac{\phi}{h}} + \frac{4(e_1 + \bfi e_2)_i (e_1 + \bfi e_2)_j}{h^2}(z-z_0)^2 e^{\frac{\phi}{h}}, 
\end{equation*}
for all $i,j\in\{1,\cdots,n\}$, \eqref{eq:integral-identity} becomes 
\begin{equation*}
\int_{\mR^2} \sum_{i,j=1}^{2}\theta_{ij}^2 e^{\frac{\phi - \bar{\phi}}{h}}\left(\begin{aligned}
& \frac{2(e_1 + \bfi e_2)_{i}(e_1 + \bfi e_2)_{j}\overline{(z-z_0)}}{h} \\ 
&\quad + \frac{4(e_1 + \bfi e_2)_i (e_1 + \bfi e_2)_j}{h^2}|z-z_0|^{2}(z-z_0)
\end{aligned}\right) = 0.
\end{equation*} 
Applying \Cref{Stationary phase}, we obtain 
\begin{equation*}
\begin{aligned}
0 &= \sum_{i,j=1}^{2} (e_1 + \bfi e_2)_i(e_1 + \bfi e_2)_j P^2(\overline{(z-z_0)}\theta_{ij}^2)(z_0) \\
&\quad + \sum_{i,j=1}^{2} 4 (e_1 + \bfi e_2)_i(e_1 + \bfi e_2)_j\frac{1}{3!}P^3(|z-z_0|^2(z-z_0) \theta_{ij}^2)(z_0) 
\end{aligned}
\end{equation*}
where $P = -\frac{1}{4\bfi} \frac{\partial^2}{\partial x \partial y} = \frac{\bfi}{4}\frac{\partial^2}{\partial x \partial y}$. Hence, we get 
\begin{equation*}
(e_1 + \bfi e_2)_i (e_1 + \bfi e_2)_j (\partial_x\Delta a^2_{ij}(z_0) - \bfi \partial_y \Delta a^2_{ij}(z_0)) = 0. 
\end{equation*}
Therefore, 
\begin{equation} 
\partial_x \Delta (\theta_{11}^2 - \theta_{22}^2)(z_0) + 2\partial_y \Delta\theta_{12}^2(z_0) + \bfi (2 \partial_x \Delta \theta_{12}^2(z_0) - \partial_y \Delta (\theta_{11}^2 - \theta_{22}^2)(z_0)) = 0  \label{2(aa)}
\end{equation}
for all $z_0\in\Omega$. 

\medskip 

\noindent \emph{Finally, we choose $u(z) = e^{\frac{- \overline{\phi}}{h}}$ and $v(z) = (z-z_0) e^{\frac{{\phi}}{h}}$ into \eqref{eq:integral-identity}.} Since 
\begin{equation*}
\partial_i u = \frac{-2(e_1 - \bfi e_2)_i}{h}\overline{(z-z_0)} e^{\frac{-\bar{\phi}}{h}}
\end{equation*}
and 
\begin{equation*}
\partial^2_{ij}u = \frac{-2(e_1 - \bfi e_2)_{i}(e_1 - \bfi e_2)_{j}}{h} e^{\frac{-\overline{\phi}}{h}} + \frac{4(e_1 - \bfi e_2)_i (e_1 - \bfi e_2)_j}{h^2}\overline{(z-z_0)}^2 e^{\frac{-\overline{\phi}}{h}}
\end{equation*}
for all $i,j\in\{1,\cdots,n\}$, \eqref{eq:integral-identity} becomes  
\begin{equation*}
\int_{\mathbb{R}^2} \sum_{i,j=1}^{2} \theta_{ij}^2 e^{\frac{\phi - \bar{\phi}}{h}}\left(\begin{aligned}
& \frac{-2(e_1 - \bfi e_2)_{i}(e_1 - \bfi e_2)_{j}(z-z_0)}{h} \\ 
&\quad + \frac{4(e_1 - \bfi e_2)_i (e_1 - \bfi e_2)_j}{h^2}|z-z_0|^2\overline{(z-z_0)} 
\end{aligned}\right) = 0.
\end{equation*}
Applying \Cref{Stationary phase}, we obtain 
\begin{equation*}
\begin{aligned} 
0 &= -\sum_{i,j=1}^{2} (e_1 - \bfi e_2)_i(e_1 - \bfi e_2)_jP^2((z-z_0) \theta_{ij}^2)(z_0) \\
&\quad + \sum_{i,j=1}^{2} 4 (e_1 - \bfi e_2)_i(e_1 - \bfi e_2)_j \frac{1}{3!}P^3(|z-z_0|^2 \overline{(z-z_0)} \theta_{ij}^2) 
\end{aligned} 
\end{equation*}
where $P = -\frac{1}{4\bfi} \frac{\partial^2}{\partial x \partial y} = \frac{\bfi}{4}\frac{\partial^2}{\partial x \partial y}$. 
Hence, we obtain 
\begin{equation*}
\sum_{i,j=1}^{2} (e_1 - \bfi e_2)_i (e_1 - \bfi e_2)_j (\partial_x\Delta \theta_{ij}^2(z_0) + \bfi \partial_y \Delta \theta_{ij}^2(z_0)) = 0.
\end{equation*}
Therefore, 
\begin{equation}
\partial_x \Delta (\theta_{11}^2 - \theta_{22}^2)(z_0) + 2\partial_y \Delta \theta_{12}^2(z_0) - \bfi (2 \partial_x \Delta \theta_{12}^2(z_0) - \partial_y \Delta (\theta_{11}^2 - \theta_{22}^2)(z_0)) = 0. \label{2(bb)}
\end{equation} 
for all $z_0\in\Omega$. By adding and subtracting \eqref{2(aa)} and \eqref{2(bb)}, we conclude 
\begin{equation*}
\partial_x \Delta (\theta_{11}^2 - \theta_{22}^2) = - 2\partial_y \Delta \theta_{12}^2, \quad
\partial_y \Delta (\theta_{11}^2 - \theta_{22}^2) = 2 \partial_x \Delta \theta_{12}^2 \quad \text{in $\Omega$.}
\end{equation*}
From the above equations, we deduce that $\theta_{11}^{2}=\theta_{22}^{2}$ and $\theta_{12}^{2}=0$. Since ${\rm tr}\,(\theta^2)=0$, it follows that $\theta^{2}=0$, completing the proof. 
\end{proof}

\appendix
\crefalias{section}{appendix}

\section{Complex geometric optics solutions\label{sec:CGO-proof}} 

The main purpose of this appendix is to refine the complex geometric optics (CGO) solutions for the equation $(-\Delta)^{k}u=0$ in a bounded smooth domain $\Omega\subset\mR^{n}$, with $n\ge 3$ and $k=1,2$, constructed in \cite[Lemma~A.4]{SS_linearized}. For any parameter $h>0$ and a nonnegative integer $m$, we define the semiclassical norm 
\begin{equation*}
\norm{u}_{H_{[h]}^{m}(\Omega)}^{2} = \sum_{\abs{\alpha}_{1}\le m} \norm{(h\partial)^{\alpha}u}_{L^{2}(\Omega)}^{2} 
\end{equation*}
where $\abs{\alpha}_{1}=\alpha_{1}+\cdots+\alpha_{n}$ for each multi-index $\alpha\in\mZ_{\ge 0}^{n}$. To make this paper self-contained, we recall the following general definition, although it is not strictly necessary: 

\begin{definition}[\cite{KEN}, see also {\cite[Definition~A.1]{SS_linearized}}]\label{def:limiting-carleman-weight}
Let $h>0$ be a given parameter, referred as the semi-classical parameter. A function $\varphi:\overline{\Omega}\rightarrow\mR$ is called a \emph{limiting Carleman weight} for the semi-classical conjugated Laplacian $P_{0,\varphi}:=e^{\frac{\varphi}{h}}(-h^{2}\Delta)e^{-\frac{\varphi}{h}}$ if the following conditions hold: 
\begin{itemize}
\item there exists an open set $\Omega_{0}\supset\overline{\Omega}$ such that $\varphi\in C^{\infty}(\Omega_{0})$; 
\item $\abs{\nabla\varphi}\neq 0$ in $\overline{\Omega}$; and 
\item $\{\Re(p_{0,\varphi}),\Im(p_{0,\varphi})\}(x,\xi)=0$ for all $(x,\xi)\in\Omega_{0}\times(\mR^{n}\setminus\{0\})$ with $p_{0,\varphi}(x,\xi)=0$, 
\end{itemize}
where $\{\cdot,\cdot\}$ denotes the Poisson bracket and 
\begin{equation*}
p_{0,\varphi}(x,\xi)=\abs{\xi}^{2}-\abs{\nabla\varphi(x)}^{2} + 2\bfi\xi\cdot\nabla\varphi(x) 
\end{equation*}
is the semi-classical principal symbol of $P_{0,\varphi}$. 
\end{definition}

\begin{example}
Standard examples of such functions $\varphi$ as described in \Cref{def:limiting-carleman-weight} include linear weights $\varphi(x)=\alpha\cdot x$ with $0\neq\alpha\in\mR^{n}$, and logarithmic weights $\varphi(x)=\log\abs{x-x_{0}}$ with $x_{0}\notin\overline{\Omega_{0}}$. 
\end{example}

We now recall an existence result in \cite{SS_linearized}: 

\begin{lemma}[{\cite[Proposition~A.3]{SS_linearized}}]\label{lem:existence}
Let $k\in\mN$. For any $v\in L^{2}(\Omega)$, for all sufficiently small $h>0$ and for all limiting Carleman weight $\varphi$ as described in \Cref{def:limiting-carleman-weight}, there exists $u\in H_{[h]}^{2k}(\Omega)$ such that 
\begin{equation*}
e^{-\frac{\varphi}{h}}(-\Delta)^{k}e^{\frac{\varphi}{h}}u=v \text{ in $\Omega$} \quad \text{satisfying} \quad \norm{u}_{H_{[h]}^{2k}(\Omega)}\le Ch^{k}\norm{v}_{L^{2}(\Omega)}
\end{equation*} 
for some positive constant $C$ independent of $h$ (but depends on $k$). 
\end{lemma}

In \cite[Lemma~A.4]{SS_linearized}, the authors construct CGO solutions under the assumption $p_{0,\varphi}(x,\nabla\psi)=0$, which in turn implies that 
\begin{equation}
\abs{\nabla\varphi}=\abs{\nabla\psi} \quad \text{and} \quad \nabla\varphi\cdot\nabla\psi=0 \quad \text{in $\Omega$.} \label{eq:conjugate-weight}
\end{equation}
We are now ready to prove the following proposition, which can be regarded as a refinement of \cite[Lemma~A.4]{SS_linearized}: 

\begin{proposition}\label{prop:CGO-general}
Let $h>0$ be a sufficiently small parameter, and let $\varphi$ be a limiting Carleman weight as in \Cref{def:limiting-carleman-weight}. Choose a real-valued function $\psi\in C^{\infty}(\overline{\Omega})$ so that \eqref{eq:conjugate-weight} holds, and define the associated transport operator 
\begin{equation*}
T:=2\nabla(\varphi+\bfi\psi)\cdot\nabla + \frac{1}{h}\Delta(\varphi+\bfi\psi). 
\end{equation*}
\begin{enumerate}
\renewcommand{\labelenumi}{\theenumi}
\renewcommand{\theenumi}{\rm (\alph{enumi})} 
\item\label{itm:CGO-harmonic} For each $m\in\mN$, let $a_{0},\cdots,a_{m-1}\in C^{\infty}(\overline{\Omega})$ solve 
\begin{equation}
Ta_{j} = -\Delta a_{j-1} \text{ in $\Omega$ for all $j=0,\cdots,m-1$} \label{eq:recurrence-PDE-a}
\end{equation}
with the convention $a_{-1}\equiv 0$. Then there exists a remainder term $r(\cdot;h)\in H_{[h]}^{2}(\Omega)$ such that 
\begin{equation*}
w(x,h)=e^{\frac{\varphi+\bfi\psi}{h}}(A_{m}(x;h)+r(x;h)) \quad \text{with} \quad A_{m}(x;h) = \sum_{j=0}^{m-1}h^{j}a_{j}(x) 
\end{equation*}
is harmonic and satisfies 
\begin{equation*}
\norm{r(\cdot;h)}_{H_{[h]}^{2}(\Omega)} \le Ch^{m}
\end{equation*}
for some positive constant $C$ independent of $h$ and $m$, but depends on $\Omega$, $\varphi$, $\psi$ and $a_{m-1}$. 
\item\label{itm:CGO-bi-harmonic} For each $m\in\mN$, let $b_{0},\cdots,b_{m-1}\in C^{\infty}(\overline{\Omega})$ solve 
\begin{equation}
T^{2}b_{j} = -(\Delta T+T\Delta)b_{j-1} - \Delta^{2}b_{j-2} \text{ in $\Omega$ for all $j=0,\cdots,m-1$} \label{eq:recurrence-PDE-b}
\end{equation}
with the convention $b_{-2}\equiv b_{-1}\equiv 0$. Then there exists a remainder term $\tilde{r}(\cdot;h)\in H_{[h]}^{4}(\Omega)$ such that 
\begin{equation*}
\tilde{w}(x,h)=e^{\frac{\varphi+\bfi\psi}{h}}(B_{m}(x;h)+\tilde{r}(x;h)) \quad \text{with} \quad B_{m}(x;h) = \sum_{j=0}^{m-1}h^{j}b_{j}(x) 
\end{equation*}
is biharmonic and satisfies 
\begin{equation*}
\norm{\tilde{r}(\cdot;h)}_{H_{[h]}^{4}(\Omega)} \le Ch^{m}
\end{equation*}
for some positive constant $C$ independent of $h$ and $m$, but depends on $\Omega$, $\varphi$, $\psi$ $b_{m-2}$ and $b_{m-1}$.  
\end{enumerate}
\end{proposition}

\begin{remark}
It is well known that the equations in \eqref{eq:recurrence-PDE-a} and \eqref{eq:recurrence-PDE-b} admit smooth solutions, see, for instance, \cite{DOS}. 
\end{remark}

\begin{proof}[Proof of \Cref{prop:CGO-general}\ref{itm:CGO-harmonic}]
First, applying \eqref{eq:recurrence-PDE-a}, we compute that 
\begin{equation}
e^{-\frac{\varphi+\bfi\psi}{h}}\Delta\left(e^{\frac{\varphi+\bfi\psi}{h}}A_{m}(\cdot;h)\right) = h^{m-1}\Delta a_{m-1} \quad \text{in $\Omega$.} \label{eq:principal1}
\end{equation}
We now apply \Cref{lem:existence} with $k=1$ and $v=e^{\frac{\bfi\psi}{h}}h^{m-1}\Delta a_{m-1}$ to construct a function $r_{0}(\cdot;h)\in H_{[h]}^{2}(\Omega)$ satisfying 
\begin{equation*}
-e^{-\frac{\varphi}{h}}\Delta \left( e^{\frac{\varphi}{h}}r_{0}(\cdot;h) \right) = e^{\frac{\bfi\psi}{h}} h^{m-1}\Delta a_{m-1} \quad \text{in $\Omega$}
\end{equation*}
and 
\begin{equation*}
\norm{r_{0}(\cdot;h)}_{H_{[h]}^{2}(\Omega)} \le Ch^{m}\norm{\Delta a_{m-1}}_{L^{2}
(\Omega)}
\end{equation*}
for some positive constant $C$ independent of both $h$ and $m$, but depends on $\varphi$. We now note that $r(\cdot;h):=e^{-\frac{\bfi\psi}{h}}r_{0}(\cdot;h)\in H_{[h]}^{2}(\Omega)$ satisfies 
\begin{equation}
-e^{-\frac{\varphi+\bfi\psi}{h}}\Delta\left( e^{\frac{\varphi+\bfi\psi}{h}}r(\cdot;h) \right) = h^{m-1}\Delta a_{m-1} \quad \text{in $\Omega$} \label{eq:remainder1} 
\end{equation}
and 
\begin{equation*}
\norm{r(\cdot;h)}_{H_{[h]}^{2}(\Omega)} \le Ch^{m}\norm{\Delta a_{m-1}}_{L^{2}
(\Omega)}
\end{equation*}
for some positive constant $C$ independent of both $h$ and $m$, but depends on $\varphi$ and $\psi$. Finally, the result follows from $\eqref{eq:principal1}-\eqref{eq:remainder1}$. 
\end{proof}

\begin{proof}[Proof of \Cref{prop:CGO-general}\ref{itm:CGO-bi-harmonic}]
Owing to the choice of $\varphi$ and $\psi$, we have $\nabla(\varphi+\bfi\psi)\cdot\nabla(\varphi+\bfi\psi)=0$ in $\Omega$. Consequently, 
\begin{equation*}
e^{-\frac{\varphi+\bfi\psi}{h}}\Delta^{2}\left(e^{\frac{\varphi+\bfi\psi}{h}}B_{m}(\cdot;h)\right) = \left( \frac{1}{h}T+\Delta \right)^{2}B(\cdot;h) \quad \text{in $\Omega$,} 
\end{equation*}
see \cite[(A.2)]{SS_linearized}. Applying \eqref{eq:recurrence-PDE-b}, we compute that 
\begin{equation}
\begin{aligned}
& e^{-\frac{\varphi+\bfi\psi}{h}}\Delta^{2}\left(e^{\frac{\varphi+\bfi\psi}{h}}B_{m}(\cdot;h)\right) \\
& \quad = h^{m-2} \left( (T\Delta+\Delta T)b_{m-1} + \Delta^{2}b_{m-2} \right) + h^{m-1}\Delta^{2}b_{m-1} \quad \text{in $\Omega$.} 
\end{aligned}
\end{equation}
We now apply \Cref{lem:existence} with $k=2$ and $v=e^{\frac{\bfi\psi}{h}}(h^{m-2} \left( (T\Delta+\Delta T)b_{m-1} + \Delta^{2}b_{m-2} \right) + h^{m-1}\Delta^{2}b_{m-1})$ to construct a function $\tilde{r}_{0}(\cdot;h)\in H_{[h]}^{4}(\Omega)$ satisfying 
\begin{equation}
\begin{aligned}
& -e^{-\frac{\varphi}{h}}\Delta^{2} \left( e^{\frac{\varphi}{h}}\tilde{r}_{0}(\cdot;h) \right) \\ 
& \quad = e^{\frac{\bfi\psi}{h}} h^{m-2} \left( (T\Delta+\Delta T)b_{m-1} + \Delta^{2}b_{m-2} \right) + h^{m-1}\Delta^{2}b_{m-1}
\end{aligned} \label{eq:principal2}
\end{equation}
and 
\begin{equation*}
\begin{aligned}
\norm{\tilde{r}_{0}(\cdot;h)}_{H_{[h]}^{4}(\Omega)} \le Ch^{m}\norm{(T\Delta+\Delta T)b_{m-1}+\Delta^{2}b_{m-2} + h\Delta b_{m-1}}_{L^{2}(\Omega)} 
\end{aligned}
\end{equation*}
for some positive constant $C$ independent of both $h$ and $m$, but depends on $\varphi$. We now note that $\tilde{r}(\cdot;h):=e^{-\frac{\bfi\psi}{h}}\tilde{r}_{0}(\cdot;h)\in H_{[h]}^{2}(\Omega)$ satisfies 
\begin{equation}
\begin{aligned}
& -e^{-\frac{\varphi+\bfi\psi}{h}}\Delta\left( e^{\frac{\varphi+\bfi\psi}{h}}\tilde{r}(\cdot;h) \right) \\ 
& \quad = h^{m-2} \left( (T\Delta+\Delta T)b_{m-1} + \Delta^{2}b_{m-2} \right) + h^{m-1}\Delta^{2}b_{m-1} \quad \text{in $\Omega$} 
\end{aligned} \label{eq:remainder2} 
\end{equation}
and 
\begin{equation*}
\norm{r(\cdot;h)}_{H_{[h]}^{2}(\Omega)} \le Ch^{m}\norm{(T\Delta+\Delta T)b_{m-1}+\Delta^{2}b_{m-2} + h\Delta b_{m-1}}_{L^{2}(\Omega)} 
\end{equation*}
for some positive constant $C$ independent of both $h$ and $m$, but depends on $\varphi$ and $\psi$. Finally, the result follows from $\eqref{eq:principal2}-\eqref{eq:remainder2}$. 
\end{proof}

\section*{Acknowledgments}
RSJ is partially supported by the NSFC grant W2431006.
PZK is supported by the National Science and Technology Council of Taiwan (NSTC 112-2115-M-004-004-MY3), and by the National Center for Theoretical Sciences of Taiwan. 
SKS is supported by IIT Bombay seed grant (RD/0524-IRCCSH0-021) and ANRF Early Career Research Grant (ECRG) (RD/0125- ANRF000-016).

\bibliographystyle{custom}
\bibliography{ref}

\end{sloppypar}
\end{document}